\documentclass[oneside,english]{amsart}
\usepackage[latin9]{inputenc}
\usepackage{color}
\usepackage{babel}
\usepackage{algorithm2e}
\usepackage{amsthm}
\usepackage{amssymb}
\usepackage{stackrel}
\usepackage{graphicx}
\usepackage[unicode=true,pdfusetitle,
 bookmarks=true,bookmarksnumbered=false,bookmarksopen=false,
 breaklinks=false,pdfborder={0 0 0},pdfborderstyle={},backref=false,colorlinks=true]
 {hyperref}
\hypersetup{
 citecolor=blue}

\makeatletter
\numberwithin{equation}{section}
\numberwithin{figure}{section}
\theoremstyle{plain}
\newtheorem{thm}{\protect\theoremname}[section]
\theoremstyle{definition}
\newtheorem{defn}[thm]{\protect\definitionname}
\theoremstyle{remark}
\newtheorem*{acknowledgement*}{\protect\acknowledgementname}
\theoremstyle{remark}
\newtheorem{rem}[thm]{\protect\remarkname}
\theoremstyle{plain}
\newtheorem{cor}[thm]{\protect\corollaryname}
\theoremstyle{plain}
\newtheorem{prop}[thm]{\protect\propositionname}

\usepackage{geometry}
\usepackage{amsmath}
\usepackage{amsthm}

\numberwithin{equation}{section}
\numberwithin{figure}{section}
\usepackage{enumitem}		
 \let\footnote=\endnote
\@ifundefined{lettrine}{\usepackage{lettrine}}{}

\theoremstyle{theorem}
\newtheorem{thmx}{Theorem}

\def\s{\sigma}

\def\SI{\Sigma}
\def\G{\Gamma}

\def\g{\gamma}

\def\L{\Lambda}

\def\dd{\mathrm{d}}

\def\D{\mathbb{D}}
\def\R{\mathbb{R}}
\def\C{\mathbb{C}}
\def\H{\mathbb{H}}
\def\N{\mathbb{N}}
\def\Z{\mathbb{Z}}

\def\ep{\varepsilon}

\def\cal{\mathcal}

\address{Department of Mathematics, The University of Chicago, Chicago, IL 60637, USA} 
\email{lkao@math.uchicago.edu}

\keywords{}

\subjclass[2000]{}

\usepackage{dsfont}

\usepackage{enumitem}

\theoremstyle{theorem}
\newtheorem{DEFN-THM}[thm]{Definition/Theorem}

\def\s{\sigma}

\def\SI{\Sigma}
\def\G{\Gamma}

\def\g{\gamma}

\def\L{\Lambda}

\def\dd{\,\mathrm{d}}

\def\D{\mathbb{D}}
\def\R{\mathbb{R}}
\def\C{\mathbb{C}}
\def\H{\mathbb{H}}
\def\N{\mathbb{N}}
\def\Z{\mathbb{Z}}

\def\ep{\varepsilon}

\def\cal{\mathcal}

\def\psl{\mathrm{PSL}}

\def\vbdy{\partial_{\infty}}

\def\TMS{\Sigma^{+}}
\def\psltr{\mathrm{PSL}(2,\R)}

\def\vp{\varphi}

\makeatother

\providecommand{\acknowledgementname}{Acknowledgement}
\providecommand{\corollaryname}{Corollary}
\providecommand{\definitionname}{Definition}
\providecommand{\propositionname}{Proposition}
\providecommand{\remarkname}{Remark}
\providecommand{\theoremname}{Theorem}

\begin{document}

\title[Pressure Metrics for Teichm\"uller Spaces of Punctured Surfaces ]{Pressure Metrics and Manhattan Curves for Teichm\"uller Spaces
of Punctured Surfaces }

\author{Lien-Yung Kao}

\begin{abstract}
In this paper, we extend the construction of pressure metrics to Teichm\"uller
spaces of surfaces with punctures. This construction recovers Thurston's
Riemannian metric on Teichm\"uller spaces. Moreover, we prove the
real analyticity and convexity of Manhattan curves of finite area
type-preserving Fuchsian representations, and thus we obtain several related
entropy rigidity results. Lastly, relating the two topics mentioned
above, we show that one can derive the pressure metric by varying
Manhattan curves.
\end{abstract}

\maketitle

\tableofcontents{}

\section{Introduction\label{sec:Introduction}}

Let $S=S_{g,n}$ be an orientable surface of genus $g$ and $n$ punctures
with negative Euler characteristic. In this paper, we discuss how
one can characterize Fuchsian representations and the geometry of
${\cal T}(S)$ the Teichm\"uller space of $S$ by studying dynamics
objects associated with them. For example, we prove rigidity results
via examining the shape of Manhattan curves, and we construct a Riemannian
metric on ${\cal T}(S)$ by derivatives of pressure. 

When $S$ has no puncture, results in this work are not new. Manhattan
curves and rigidity results are, for instance, discussed in \cite{Burger:1993wb,Sharp:1998if},
and the pressure metric on ${\cal T}(S)$ is discovered in \cite{McMullen:2008eh}
and further investigated in \cite{Pollicott:2014uk,Bridgeman:2017jy}.
Nevertheless, when $S$ has punctures, especially when Fuchsian representations
are not convex co-compact, much less results along this line are proved.
Indeed, in such cases, their dynamics are much more complicated because
the presence of parabolic elements. 

Using a similar idea in \cite{Ledrappier:2008wq,Kao:2018th}, we study
geodesics flows over hyperbolic surfaces with cusps by countable state
Markov shifts and corresponding suspension flows. Notice that for
countable state Markov shifts, different from compact cases, for unbounded
potentials without sufficient control of their regularity and values
around cusps, the pressure of their perturbation might not only lose
the analyticity but also information of some thermodynamics data.
For example, time changes for suspension flows over a non-compact
Markov shift may not take equilibrium states to equilibrium states
for some potentials (cf. \cite{Cipriano:2018wz}). 

To overcome these issues, we carefully study the associated geometric
potential (or the roof function of the suspension flow). By doing
so, we know exactly where the pressure function (of geometric potentials
and their weighted sums) is analytic. Thus, we can mimic the procedure
used in compact cases. More precisely, we derive a version of Bowen's
formula which relating the topological entropy of the geodesic flow
and the corresponding roof function. With Bowen's formula and the
analyticity of pressure, we prove the convexity of Manhattan curves,
and using the second derivative of pressure we construct a Riemannian
metric on ${\cal T}(S)$. 

To put our results in context, we now introduce necessary notations
and definitions. Recall that a representation $\rho\in{\rm Hom}(\pi_{1}S,\psltr)$
is \textit{Fuchsian} if it is discrete and faithful, and $\rho$ has
\textit{finite area} if the hyperbolic surface $X_{\rho}=\rho(\pi_{1}S)\backslash\H$
has finite area. We say two finite area Fuchsian representations $\rho_{1},\rho_{2}$
are \textit{type-preserving} if there exists an isomorphism $\iota:\rho_{1}(\pi_{1}S)\to\rho_{2}(\pi_{1}S)$
sending parabolic elements to parabolic elements and hyperbolic elements
to hyperbolic elements. Here ${\rm PSL}(2,\R)$ refers to the space
of orientation preserving isometries of the hyperbolic plane $\H$.

Let $\rho_{1}$ and $\rho_{2}$ be two Fuchsian representations. Recall
that $d_{\rho_{1},\rho_{2}}^{a,b}$ the \textit{weighted Manhattan
metric} on $\H\times\H$ with respect to $\rho_{1},\rho_{2}$ is given
by, fixing $o=(o_{1},o_{2})$, $d_{\rho_{1},\rho_{2}}^{a,b}(o,\g o):=ad(o_{1},\rho_{1}(\g)o_{1})+bd(o_{2},\rho_{1}(\g)o_{2})$
for $\g\in\pi_{1}(S)$ where $d$ is the hyperbolic distance on $\H$.
Notice that we only interested in non-negative weights, i.e., $a,b\geq0$
and $ab\neq0$. We denote the associated \textit{Poincar\'e series}
by
\[
Q_{\rho_{1},\rho_{2}}^{a,b}(s):={\displaystyle \sum_{\g\in\pi_{1}(S)}e^{-s\cdot d_{\rho_{1},\rho_{2}}^{a,b}(o,\g o)}.}
\]

\begin{defn}
[Manhattan Curve]The Manhattan curve ${\cal C}(\rho_{1},\rho_{2})$
of $\rho_{1}$, $\rho_{2}$ is given by
\[
{\cal C}(\rho_{1},\rho_{2}):=\{(a,b)\in\R_{\geq0}\times\R_{\geq0}\backslash(0,0):\ \delta_{\rho_{1},\rho_{2}}^{a,b}=1\}
\]
where $\delta_{\rho_{1},\rho_{2}}^{a,b}$ is the \textit{critical
exponent} of $Q_{\rho_{1},\rho_{2}}^{a,b}(s)$, i.e., $Q_{\rho_{1},\rho_{2}}^{a,b}(s)$
is divergent if $s<\delta_{\rho_{1},\rho_{2}}^{a,b}$ and is convergent
if $s>\delta_{\rho_{1},\rho_{2}}^{a,b}$.
\end{defn}

By definition, one can regard ${\cal C}(\rho_{1},\rho_{2})$ as a
generalization of the critical exponents for $\rho_{1}$ and $\rho_{2}$.
Obviously, taking $a=0$ (respectively, $b=0$), $\delta_{\rho_{1},\rho_{2}}^{a,b}$
reduces to $\delta_{\rho_{1}}$ the classical critical exponent for
$\rho_{1}$ (respectively, $\delta_{\rho_{2}}$). By Otal and Peign\'e
\cite{Otal:2004fn}, we know $\delta_{\rho_{1}}$ is also the topological
entropy of the geodesic flow over $X_{\rho_{1}}$. 

As mentioned above, using a symbolic model given in \cite{Ledrappier:2008wq},
for every finite area Fuchsian representation $\rho$, we can code
the geodesic flow over $X_{\rho}$. Elaborated discussion of the coding
of geodesic flows is in Section \ref{sec:Geodesic-Flows-Symbolic}.
We briefly introduce the idea and strategy below. We will associate
the geodesic flow on the smaller special section $\Omega_{0}\subset T^{1}X_{\rho}$
with a suspension flow $(\TMS,\sigma,\tau_{\rho})$ where $(\TMS,\sigma)$
is a countable state Markov shift and $\tau_{\rho}:\TMS\to\R^{+}$
is the roof function. Furthermore, by the construction, the roof function
$\tau_{\rho}$ is a continuous function prescribing the length of
closed geodesics. We sometimes call $\tau_{\rho}$ the \textit{geometric
potential} of $\rho$. Moreover, one important feature of this symbolic
model is that if $\rho_{1},\rho_{2}$ are finite area type-preserving
Fuchsian representations, then they correspond to the same Markov
shift $(\TMS,\sigma)$ but to different roof functions $\tau_{\rho_{1}},\tau_{\rho_{2}}$.
In other words, we can use roof functions to characterize finite area
type-preserving Fuchsian representations.

Using this symbolic model, we can characterize ${\cal C}(\rho_{1},\rho_{2})$
as solutions of a version of Bowen's formula. Furthermore, we derive
the first main result of the paper:

\begin{thmx}\label{Thm:manhattan-strictly-convex}

Let $\rho_{1}$, $\rho_{2}$ be two finite area type-preserving Fuchsian
representations. Then ${\cal C}(\rho_{1},\rho_{2})$ is a real analytic
curve, and ${\cal C}(\rho_{1},\rho_{2})$ is strictly convex unless
$\rho_{1}$ and $\rho_{2}$ are conjugate in ${\rm PSL}(2,\R)$, in
such cases ${\cal C}(\rho_{1},\rho_{2})$ is a straight line.

\end{thmx}

Using the shape of Manhattan curve, we can further prove rigidity
results related with following dynamics quantities.

\begin{defn}
Let $\rho_{1}$, $\rho_{2}$ be a pair of Fuchsian representations.

\begin{enumerate}[font=\normalfont]

\item \textit{Bishop-Steiger entropy} $h_{BS}(\rho_{1},\rho_{2})$
of $\rho_{1}$ and $\rho_{2}$ is defined as 
\[
h_{{\rm BS}}(\rho_{1},\rho_{2}):={\displaystyle \lim_{T\to\infty}\frac{1}{T}\ln\left(\#\{[\g]\in[\pi_{1}(S)]:d(o,\rho_{1}(o,\g o))+d(o,\rho_{2}(o,\g o))\leq T\}\right)}.
\]

\item \textit{The intersection number} ${\rm I}(\rho_{1},\rho_{2})$
of $\rho_{1}$ and $\rho_{2}$ is defined as 
\[
{\rm I}(\rho_{1},\rho_{2}):=\lim_{n\to\infty}\frac{l_{2}[\g_{n}]}{l_{1}[\g_{n}]}
\]
 where $\{[\g_{n}]\}_{n=1}^{\infty}$ is a sequence of conjugacy classes
for which the associated closed geodesics $\g_{n}$ become equidistributed
on $X_{\rho_{1}}$ with respect to area.

\end{enumerate}
\end{defn}

Using a dynamics interpretation of ${\rm I}(\rho_{1},\rho_{2})$ and
the convexity and analyticity of pressure, we recover the following
results of Bishop and Steiger \cite{Bishop:1991gz}, and Thurston.

\begin{thmx}\label{thm-entropy-rigidity}

Let $\rho_{1}$, $\rho_{2}$ be a pair of area type-preserving Fuchsian
representations, we have 

\begin{enumerate}[font=\normalfont]

\item$($Bishop-Steiger Rigidity$)$ $h_{{\rm BS}}(\rho_{1},\rho_{2})\leq\frac{1}{2}$,
and the equality holds if and only if $\rho_{1}$ and $\rho_{2}$
are conjugate in $\psltr$.

\item $($The Intersection Number Rigidity$)$ ${\rm I}(\rho_{1},\rho_{2})\geq1$,
and the equality holds if and only if $\rho_{1}$ and $\rho_{2}$
are conjugate in $\psltr$.

\end{enumerate}

\end{thmx}

\begin{rem} \

\begin{enumerate}[font=\normalfont]

\item One might prove ${\cal C}(\rho_{1},\rho_{2})$ is $C^{1}$
and Theorem \ref{thm-entropy-rigidity} without employing symbolic
dynamics. Nevertheless, symbolic dynamics provides a convenient approach
to control the analyticity of pressure, and hence to prove the analyticity
of ${\cal C}(\rho_{1},\rho_{2})$.

\item It is no very clear why ${\rm I}(\rho_{1},\rho_{2})$ is well-defined.
We will justify it in Section \ref{sec:Geodesic-Flows-Symbolic}.

\item The intersection number rigidity is known as a work of Thurston
amount experts. However, due to the limited knowledge of the author,
for the non-convex co-compact cases we cannot find a reference of
it.

\end{enumerate}

\end{rem}

We now change gear from pairs of Fuchsian representations to the space
of conjugacy classes of Fuchsian representations, that is, the Teichm\"uller
space of $S=S_{g,n}$. Recall that the Teichm\"uller space of $S$
is defined as 
\[
{\cal T}(S):={\rm Hom_{tp}^{F}}(\pi_{1}(S),\psl(2,\R))/\sim
\]
where ${\rm Hom_{tp}^{F}}(\pi_{1}(S),\psl(2,\R))$ is the space of
finite area type-preserving Fuchsian representations, and $\rho_{1}\sim\rho_{2}$
if they are conjugate in $\psl(2,\R)$. 

Through the symbolic model, there is a thermodynamic mapping $\Psi:{\cal T}(S)\to{\rm \mathbf{P}}$
where ${\rm \mathbf{P}}$ is a special space of continuous functions
over $\TMS$ containing geometric potentials. Using the pressure and
variance we can define a norm $||\cdot||_{{\rm P}}$ over ${\rm \mathbf{P}}$.
Using the pullback of $||\cdot||_{{\rm P}}$, we can define a Riemannian
metric $||\cdot||_{{\rm }}$ on ${\cal T}(S)$. We call this Riemannian
metric the \textit{pressure metric.} Moreover, $||\cdot||_{{\rm }}$
can also be derived by the Hessian of the intersection number:

\begin{thmx}[The Pressure Metric]\label{thm:pressure metric}

Suppose $\rho_{t}\in{\cal T}(S)$ is an analytic path for $t\in(-\ep,\ep)$.
Then ${\rm I}(\rho_{0},\rho_{t})$ is real analytic and 
\[
||\dot{\rho}_{0}||^{2}:=||{\rm d}\psi(\dot{\rho}_{0})||_{{\rm P}}^{2}=\left.\frac{{\rm d}^{2}{\rm I}(\rho_{0},\rho_{t})}{{\rm d}t^{2}}\right|_{t=0}
\]
defines a Riemannian metric on ${\cal T}(S_{g,n})$. 

\end{thmx}

We briefly discuss the history of this Riemannian metric $||\cdot||_{{\rm }}$
on ${\cal T}(S_{g,n})$. When $n=0$, Thurston first discovered it
by using the Hessian of the intersection number. Thus, this Riemannian
metric is also known as \textit{Thurston's Riemannian metric}. Moreover,
proved by Wolpert \cite{Wolpert:1986wv}, this Riemannian metric is
exactly the Weil-Petersson metric on ${\cal T}(S_{g,0})$. McMullen
\cite{McMullen:2008eh} recovered this Riemannian metric using thermodynamic
formalism and called it the \textit{pressure metric}. Carrying over
the same spirit, Bridgeman, Canary, Labourie, and Sambarino \cite{Bridgeman:2013to}
generalized this dynamics approach and constructed a Riemannian metric
on the space of Anosov representations, i.e., a higher rank generalization
of ${\cal T}(S_{g,0})$. Our Theorem \ref{thm:pressure metric} extends
the pressure metric and Thurston's construction to ${\cal T}(S_{g,n})$
for $n>0$. 

The last result of the paper is to link the two main topics in this
work: Manhattan curves and the pressure metric. We prove that when
we look at a path in ${\cal T}(S)$, the variation of corresponding
Manhattan curves contains information of the pressure metric. Similar
result has been proved by Pollicott and Sharp \cite{Pollicott:2014uk}
when $S$ is a closed surface. We generalize it to surfaces with punctures.

\begin{thmx}\label{thm:pressure-metric-manhattan-curve}

Let $(s,\chi_{t}(s))$ be the coordinates of points on the Manhattan
curve ${\cal C}(\rho_{0},\rho_{t})$, then we have 
\[
\left.\frac{{\rm d}^{2}{\cal \chi}_{t}(s)}{{\rm d}t^{2}}\right|_{t=0}=s(s-1)\cdot||\dot{\rho}_{0}||_{{\rm }}^{2}\ \ {\rm for\ }s\in(0,1).
\]

\end{thmx}

The paper is organized as follows. In Section \ref{sec:Preliminary},
we introduce some background knowledge of geometry and thermodynamic
formalism of countable state Markov shifts. In Section \ref{sec:Geodesic-Flows-Symbolic}
we discuss the coding of geodesic flows and important properties of
the corresponding roof functions. We study the analyticity of the
pressure function in Section \ref{sec:Phase-Transitions}. Section
\ref{sec:Manhattan-and-Rigidity} is devoted to investigating the
shape of Manhattan curve and rigidity. In Section \ref{sec:The-Pressure-Metric},
we construct the pressure metric. In the last section, we focus on
the relation between Manhattan curves and the pressure metric. 

\begin{acknowledgement*}
The author is grateful to Prof. Fran\c{c}ois Ledrappier for proposing
the problem and numerous supports, to Prof. Dick Canary for many insightful
suggestions and helps. Substantial portions of this paper were written
while the author was visiting Prof. Jih-Hsin Cheng at Academia Sinica,
Taiwan. The author would like to thank Prof. Jih-Hsin Cheng and Academia
Sinica for their hospitality. The author is partially supported by
the National Science Foundation Postdoctoral Research Fellowship under
grant DMS 1703554.
\end{acknowledgement*}

\section{Preliminary\label{sec:Preliminary}}

\subsection{Geometry\label{subsec:Geometry}}

Through out this paper, $S=S_{g,n}$ is an orientable surface of $g$
genus and $n$ punctures and with negative Euler characteristic. In
this work, we are interested in finite area hyperbolic surfaces homemorphic
to $S$, that is, $S$ pair with a Riemannian metric ${\rm g}$ of
Gaussian curvature -1. Notice that every such surface $(S,{\rm g)}$
can be obtained by a Fuchsian representation. More precisely, $(S,{\rm g)}$
is isomorphic to the hyperbolic surface $X_{\rho}=\rho(\pi_{1}(S))\backslash\H$. 

For short, let us denote $\rho(\pi_{1}S)$ by $\G$. Recall that $\vbdy\H$
the \textit{boundary} of $\H$ is defined as $\R\cup\{0\}$, and $\Lambda(\G):=\overline{\{\g\cdot o:\g\in\G\}}$
denotes the \textit{limit set} of $\G$. An element $\g\in\G$ is
called \textit{hyperbolic} if $\g$ has two fixed points on $\L(\G)$,
namely, the \textit{attracting fixed point} $\g_{+}$ (i.e., $\lim_{n\to\infty}\g^{n}o=\g_{+}$)
and the \textit{repelling fixed point} $\g_{-}$ (i.e., $\lim_{n\to-\infty}\g^{n}o=\g_{-}$);
$\g$ is called \textit{parabolic} if it has one fixed point. Because
$X_{\rho}$ is negatively curved, we know that every closed geodesic
$\lambda$ on $X_{\rho}$ corresponds to a unique hyperbolic element
$\g$ (up to conjugation), and vice versa. Moreover, the length of
$\lambda$ equals to $l[\g]$ the translation distance of $\g$, that
is, $l[\g]:=\min\{d(x,\g x):x\in\H\}$.

A natural dynamical system associated to $X_{\rho}$ is the geodesic
flow $g_{t}:T^{1}X_{\rho}\to T^{1}X_{\rho}$ on the unit tangent bundle
$T^{1}X_{\rho}$, which translates many geometric problems to dynamics
problems. We recall that the \textit{Busemann function} $B:\vbdy\H\times\H\times\H$
is defined as 
\[
B_{\xi}(x,y):={\displaystyle \lim_{z\to\xi}d(x,z)-d(y,z)}
\]
for $x,y,z\in\H$ and $\xi\in\vbdy\H.$ Lift the geodesic flow $g_{t}:T^{1}X_{\rho}\to T^{1}X_{\rho}$
to its universal covering $T^{1}\H$, by abusing the notation, we
have the geodesic flow $g_{t}:T^{1}\H\to T^{1}\H.$\textcolor{red}{{} }

Recall that two Fuchsian representations $\rho_{1}$, $\rho_{2}$
are\textit{ type-preserving} if there exists an isomorphism $\iota:\rho_{1}(\pi_{1}S)\to\rho_{2}(\pi_{1}S)$
such that $\iota$ sends hyperbolic elements to hyperbolic elements
and parabolic elements to parabolic elements. The following theorem
indicates that if $\rho_{1},\rho_{2}$ are type-preserving finite
area Fuchsian representations, then we can link $X_{\rho_{1}}$ and
$X_{\rho_{2}}$ is a controlled manner.

\begin{thm}
[Fenchel-Nielsen Isomorphism Theorem; \cite{Kapovich:2009fk}, Theorem 5.5, 8.16, 8.29]\label{thm:Fenchel-Neilsen-Thm}
Suppose $\rho_{1},\rho_{2}$ are two finite area type-preserving Fuchsian
representations of $\pi_{1}S$. Then there exists an bilipschitz homeomorphism
${\rm b}:X_{\rho_{1}}\to X_{\rho_{2}}$. Moreover, one can extend
${\rm b}$ to an equivarient bilipschitz map, abusing the notation,
${\rm b}:\vbdy\H\cup\H\to\vbdy\H\cup\H$.
\end{thm}

\begin{rem}
\label{rem:bilipschitz}In \cite{Kapovich:2009fk}, the homeomorphism
${\rm b}:X_{\rho_{1}}\to X_{\rho_{2}}$ is stated to be quasiconformal.
Nevertheless, using Mori's Theorem (cf. p.30 \cite{Ahlfors:2006ix})
it is not hard to see that quasiconformal homeomorphisms are indeed
bilipschitz maps. 
\end{rem}

In the following, we state a special case of \cite[Theorem A]{Kim:2001km}.

\begin{thm}
[Marked Length Spectrum Rigidity]\label{thm:proportinal marked length spectrum}Let
$\rho_{1},\rho_{2}:\pi_{1}(S)\to{\rm PSL}(2,\R)$ be Zariski dense
Fuchsian representations. Suppose $\rho_{1},\rho_{2}$ have the same
marked length spectrum, that is, $l[\rho_{1}(\gamma)]=k\cdot l[\rho_{2}(\gamma)]$
for some $k>0$ and for sufficiently many yet finite $\g\in\pi_{1}(S)$.
Then $\rho_{1}$ and $\rho_{2}$ are conjugate in ${\rm PSL}(2,\R)$. 
\end{thm}

\begin{rem}
\ \begin{enumerate}[font=\normalfont]

\item A representation $\rho:\pi_{1}(S)\to{\rm PSL}(2,\R)$ is called
\textit{Zariski dense} if it is irreducible and $\rho(\pi_{1}(S))$
has no global fixed point on $\vbdy\H$. It is clear that finite area
Fuchsian representations are Zariski dense. 

\item The above result should be known before Kim, and \cite[Theorem A]{Kim:2001km}
is much more general than the above one. Nevertheless for the convenience
we quote \cite[Theorem A]{Kim:2001km}. 

\end{enumerate}
\end{rem}

\subsection{Countable State Markov Shifts\label{subsec:Countable-Markov}}

In this subsection we aim to introduce terminologies of thermodynamic
formalism for countable state (topological) Markov shifts. Reader
can find more details in Mauldin's and Urba\'nski's book \cite{Mauldin:2003dn}
and Sarig's note \cite{Sarig:2009wta}. 

Let ${\cal A}$ a countable set and $\mathbb{A}=(t_{ij})_{{\cal A}\times{\cal A}}$
be a matrix of zeros and ones with no columns or rows are all zeros.

\begin{defn}
[Countable State Markov Shift] The (one-sided) countable state Markov
shift with set of \textit{alphabet} (or states) ${\cal {\cal A}}$
and \textit{transition matrix} $\mathbb{A}$ is defined by 
\[
\TMS_{\mathbb{A}}:=\{x=(x_{i})\in{\cal A}^{\N}:t_{x_{n}x_{n+1}}=1\ {\rm \forall}n\in\N\}
\]
equipped with the topology generated by the collection of \textit{cylinders}
\[
[a_{0},...,a_{n}]:=\{x\in\TMS_{\mathbb{A}}:x_{i}=a_{i},0\leq i\leq n\}\ \ (n\in\N,a_{0},...,a_{n}\in{\cal A})
\]
and coupled the the left (shift) map $\sigma:(x_{0},x_{1},x_{2}...)\mapsto(x_{1},x_{2},...)$. 
\end{defn}

A \textit{word of length} $n$ on an alphabet $\mathcal{A}$ is a
finite sequence $(a_{0},a_{1},...,a_{n-1})\in{\cal A}^{n-1}$ for
all $n\in\N\backslash\{0\},$ and a word $(a_{0},a_{1},...,a_{n-1})$
is \textit{admissible} with respect to $\mathbb{A}=(t_{ab})_{\mathcal{A}\times\mathcal{A}}$
if $t_{a_{i}a_{j}}=1$. 

From now on we will omit the subscript $\mathbb{A}$ from $\TMS_{\mathbb{A}}$
and simply use $\TMS$ for one-sided Markov shifts because our discussion
here only focus on a fixed transition matrix. 

Recall that a Markov shift $(\TMS,\sigma)$ is\textit{ topologically
transitive} if for all $a,b\in\mathcal{A}$ there exists an admissible
word $(a,...,b)$, and is \textit{topological mixing} if for all $a,b\in\mathcal{A}$
there exists a number $N_{ab}$ such that for all $n\geq N_{ab}$
there exists an admissible word $(a,...,b)$ of length $n$.

Let $g:\TMS\to\R$ be a function. For $n\geq1,$ the $n$-th \textit{variation}
of $g$ is defined by 
\[
{\rm V}_{n}(g):=\sup\{|g(x)-g(y)|:\ x,y\in\TMS,x_{i}=y_{i}\ {\rm for\ }0\leq i\leq n-1\}.
\]
When $\sum_{n=0}^{\infty}{\rm V}_{n}(g)<\infty$ we say that $g$
has \textit{summable variations,} and in particular, we call $g$
a \textit{locally H\"{o}lder continuous function} if there exists
$C>0$ and $\theta\in(0,1)$ such that ${\rm V}_{n}(g)\leq C\cdot\theta^{n}$
for $n\geq1.$ 

We remark that when the set of alphabet ${\cal A}$ is finite the
Markov shift is called a \textit{subshift of finite type}, and in
that case $\TMS$ is a compact set. When $\mathcal{A}$ is infinite
$\TMS$ is no longer compact.  Nevertheless, countable state Markov
shifts with the following property can be studied similarly as in
the compact cases. 

\begin{defn}
[BIP] We say $(\TMS_{\mathbb{A}},\sigma)$ has the \textit{big image
and preimages (BIP) property }if there exists a finite collection
of states $s_{1},s_{2},...,s_{n}\in\mathcal{A}$ such that for every
state $s\in\mathcal{A}$ there are some $i,j\in\{1,2,..,n\}$ such
that $(s_{i},s)$, $(s,s_{j})$ are admissible. 
\end{defn}

\begin{defn}
[Topological Pressure for Countable State Markov Shifts] Let $(\TMS,\s)$
be a topologically mixing Markov shifts and $g:\TMS\to\R$ has summable
variations. The \textit{topological pressure} (or the \textit{Gurevich
pressure}) of $g$ is defined by 
\[
P(g):=\lim_{n\to\infty}\frac{1}{n}\log\sum_{x\in{\rm Fix}^{n}}e^{S_{n}g(x)}\mathds1{}_{[a]}(x),
\]
where ${\rm Fix}^{n}:=\{x\in\TMS:\sigma^{n}(x)=x\}$, $a\in\mathcal{A}$
is any state, and $S_{n}g(x)=g(x)+...+g(\s^{n-1}(x))$ is the $n$-th
ergodic sum of $g$. 
\end{defn}

Notice that the topological pressure is independent on the state $a\in\mathcal{A}$
(cf. \cite{Sarig:2009wta}). 

\begin{thm}
[Variational Principle; \cite{Sarig:1999wo} Theorem 3] Let $(\TMS,\s)$
be a topologically mixing Markov shifts and $g:\TMS\to\R$ has summable
variations. If $\sup g<\infty$ then 
\[
P(g)=\sup\{h_{\sigma}(m)+\int_{\TMS}g\dd m:\ m\in{\cal M_{\s}}\ {\rm and}\ -\int_{\TMS}g\dd m<\infty\}
\]
where $h_{\s}(m)$ is the measure theoretic entropy of $m$ and ${\cal M}_{\s}$
is the set of $\sigma-$invariant Borel probability measures on $\TMS$. 
\end{thm}

We want to remark that although Mauldin and Urba\'{n}ski, and Sarig
defined countable state Markov shifts and the topological pressure
differently. However, when the Markov shift is topologically mixing
and has the BIP property, their definition are the same (cf. \cite[Section 7]{Mauldin:2001dn}).
Since in this paper we only focus on topologically mixing Markov shifts
with the BIP property, we will use both results from Mauldin and Urba\'{n}ski,
and Sarig.

Recall that a measure $m\in\mathcal{M_{\s}}$ is called an \textit{equilibrium
state} for $g$ if $P(g)=h_{\s}(m)+\int g\dd m$. A measure $\nu\in{\cal M}_{\s}$
is called a \textit{Gibbs measure} for $g$ if there exists constants
$G>1$ and $P$ such that for all cylinder $[a_{0},...,a_{n-1}]$
and for very $x\in[a_{0},...,a_{n-1}]$ we have 
\[
\frac{1}{G}\leq\frac{\nu[a_{0},a_{1},...,a_{n-1}]}{\exp[S_{n}g(x)-nP]}\leq G.
\]

\begin{rem}
We would like to point out that there are subtle differences between
Gibbs states and equilibrium states. Every equilibrium state is a
Gibbs state but not vice versa. More precisely, if $g$ is locally
H\"older with finite pressure and $\sup g<\infty$. Then $g$ has
a unique Gibbs measure $\nu_{g}$, and $g$ has at most one equilibrium
state. Furthermore, with the additional condition $-\int g\dd\nu_{g}<\infty$,
we know the unique Gibbs state $\nu_{g}$ is the equilibrium state
for $g$ (cf. \cite[Theorem 4.5, 4.6, 4.9]{Sarig:2009wta} and \cite[Theorem 2.2.4, 2.2.9]{Mauldin:2003dn}). 
\end{rem}

Two functions $f,g:\TMS\to\R$ are \textit{cohomologus}, denoted by
$f\sim g$, if there exists a function $h:\TMS\to\R$ such that $f=g+h-h\circ\sigma$
where $h$ is called a \textit{transition function}. The following
theorem shows that the thermodynamic data are invariant in each cohomologus
class of locally H\"older continuous functions.

\begin{thm}
\cite[Theorem 2.2.7]{Mauldin:2003dn} \label{thm:props-of-Gibbs-measures}Suppose
$(\TMS,\sigma)$ is topologically mixing, and $f,g:\TMS\to\R$ are
locally H\"older continuous function with Gibbs measures $\nu_{f}$
and $\nu_{g}$, respectively. Then the following are equivalent: \begin{enumerate}[font=\normalfont]

\item $\nu_{f}=\nu_{g}$.

\item$($\rm{Liv\v{s}ic Theorem}$)$ There exists a constant $R>0$
such that $\forall$ $n\geq1$ and $x\in{\rm Fix}^{n}$ we have $S_{n}f(x)-S_{n}g(x)=nR$. 

\item $f-g$ is cohomologus to a constant $R$ via a bounded H\"older
continuous transition function.\\
Moreover, when above assertions are true, then $R=P(f)-P(g)$.

\end{enumerate}
\end{thm}

We remark that we can define a two-sided countable state Markov shift
$\Sigma_{\mathbb{A}}$ as 
\[
\SI_{\mathbb{A}}:=\{x=(x_{i})\in{\cal A}^{\Z}:t_{x_{n}x_{n+1}}=1\ {\rm \forall}n\in\Z\}
\]
and define similarly all the thermodynamics data. Notice that if a
potential on a two-sided shift space $(\Sigma,\sigma)$ is only depending
on its future coordinate, then to understand the associated thermodynamics
data, it is sufficient to study its behavior on the one-sided shift
$(\TMS,\sigma)$. For a two-sided sequence $(...,a,\dot{b},c,...)$,
$\dot{b}$ means $b$ is at the zero-th coordinate, i.e., $a=x_{-1},b=x_{0},c=x_{1}$. 

Let $(\TMS_{\mathbb{}},\sigma)$ be a topologically mixing countable
state Markov shift with the\textit{ }BIP property. In the following,
we list a few theorems about the analyticity of pressure and phase
transition phenomena.

\begin{thm}
[Analyticity of Pressure; \cite{Mauldin:2003dn} Theorem 2.6.12 and 2.6.13, \cite{Sarig:2003hl} Corollary 4]\label{thm:analyticity-pressure}
Suppose $t\mapsto f_{t}$ is an real analytic family of locally H\"{o}lder
continuous functions for $t\in\Delta$ where $\Delta$ is an interval
of $\R$ and $P(f_{t})<\infty$ for $\Delta$. Then the pressure function
$t\mapsto P(f_{t})$, for $t\in\Delta$, is also real analytic. Moreover,
the derivative of the pressure is 
\[
\left.\frac{{\rm d}}{{\rm d}t}P(f_{t})\right|_{t=0}=\int_{\TMS}\dot{f}_{0}\dd\nu_{f_{0}},
\]
where $\nu_{f_{0}}$ is the unique Gibbs state for $f_{0}$. 
\end{thm}

\begin{thm}
[Phase Transition; \cite{Sarig:1999wo,Sarig:2001bj}, \cite{Mauldin:2003dn}]

Let $g:\TMS\to\R$ be a locally Hölder continuous function with $g>0$.
Then there exists $s_{\infty}>0$ such that 
\[
P_{\sigma}(-tg)=\begin{cases}
\infty & \mbox{if}\ t<s_{\infty},\\
\mbox{real analytic} & \mbox{if}\ t>s_{\infty},
\end{cases}
\]
where $t\mapsto P_{\sigma}(-tg)$ is the pressure function. Moreover,
$-tg$ has a unique Gibbs state $\nu_{-tg}$ for $t>s_{\infty}$.
\end{thm}

Let $f:\TMS\to\R$ be a locally H\"older continuous function and
$m\in{\cal M}_{\s}$ is an invariant measure. Recall that the \textit{variance}
${\rm Var}(f,m)$ of $f$ with respect to $m$ is defined by 
\[
{\rm Var}(f,m):=\lim_{n\to\infty}(\frac{1}{n}\int_{\TMS}(S_{n}f-\int_{\TMS}f\dd m)^{2})^{\frac{1}{2}}.
\]

Using Theorem \ref{thm:analyticity-pressure} and \cite[Theorem 5.10, 5.12]{Sarig:2009wta}
(or \cite[Theorem 2.6.14, Lemma 4.8.8]{Mauldin:2003dn}, we have the
following corollary.

\begin{cor}
[Derivatives of Pressure]\label{cor:derivative-pressure} Suppose
$f+tg$ is a family of locally H\"older continuous functions with
finite pressure for $t\in(-\ep,\ep)$. If $g$ is bounded then 
\[
P(f+tg)=P(f)+t\cdot\int_{\TMS}f\dd\nu_{f}+\frac{t^{2}}{2}\cdot{\rm Var}(g,\nu_{f})+o(t^{2})
\]
 where $\nu_{f}$ is the Gibbs measure for $f$. Moreover, ${\rm Var}(g,\nu_{f})=0$
if and only if $g$ is cohomologus to zero. 
\end{cor}

\subsection{Suspension Flows over Countable State Markov Shifts}

Let $(\TMS,\sigma)$ be a topologically mixing countable state Markov
shift with the BIP property and $\tau:\TMS\to\R^{+}$ be bounded away
from zero and locally H\"older continuous. The \textit{suspension
space} (relatively to $\tau)$ is the set 
\[
\TMS_{\tau}:=\{(x,t)\in\TMS\times\R:\ 0\leq t\leq\tau(x)\}/\sim,
\]
where $(x,\tau(x))\sim(\sigma x,0)$ for every $x\in\TMS$. The \textit{suspension
flow} $\phi_{t}$ with \textit{roof function} $\tau$ is the (vertical)
translation flow on $\TMS_{\tau}$ given by 
\[
\phi_{t}(x,s)=(x,s+t)\ \mbox{for}\ x\in\TMS{\rm \ and\ }0\leq s+t\leq\tau(x).
\]
Similarly, we can define suspension flows over a two-sided shift.

In the following, we list several equivalent definitions of the topological
pressure for suspension flows. These definitions are from Savchenko
\cite{Savchenko:1998fh}; Barreira and Iommi \cite{Barreira:2006fd};
Kempton \cite{Kempton:2011hs}; and Jaerisch and Kesseböhmer, and
Lamei \cite{Jaerisch:2014js}.

Given a $F:\TMS_{\tau}:\R$ continuous function, we define the function
$\Delta_{F}:\TMS\to\R$ by
\[
\Delta_{F}(x):=\int_{0}^{\tau(x)}F(x,t)\dd t.
\]

\begin{DEFN-THM}[Topological Pressure for Suspension Flows]\label{def:-topo pressure for sus flow}
Suppose $F:\TMS_{\tau}\to\R$ is a function such that $\Delta_{F}:\TMS\to\R$
is locally H\"older continuous. The following description of $P_{\phi}(F)$
the \textit{topological pressure} of $F$ over the suspension flow
$(\TMS_{\tau},\phi)$ are equivalent: 

\begin{alignat*}{1}
P_{\phi}(F) & =\lim_{T\to\infty}\frac{1}{T}\log\left(\sum_{\underset{0\leq s\leq T}{\phi_{s}(x,0)=(x,0)}}\exp\left(\int_{0}^{s}F(\phi_{t}(x,0))\dd t\right)\mathds1_{[a]}(x)\right)\\
 & =\sup\left\{ h_{\phi}(\mu)+\int_{\TMS_{\tau}}F\dd\mu:\ \mu\in\mathcal{M_{\phi}}\ \mbox{and }-\int_{\TMS_{\tau}}\tau\dd\mu<\infty\right\} ,
\end{alignat*}

where $a$ is any state in $\mathcal{A}$ and $\mathcal{M}_{\phi}$
is the set of $\phi-$invariant Borel probability measures on $\TMS_{\tau}$.
Moreover, if $\mu\in{\cal M_{\phi}}$ such that $P_{\phi}(F)=h_{\phi}(\mu)+\int_{\TMS_{\tau}}F\dd\mu$
then we call $\mu$ an \textit{equilibrium state} for $F$.

\end{DEFN-THM}

We finish this subsection by recalling an important observation of
relations between invariant measures on $\TMS$ and on $\TMS_{\tau}$.
\begin{thm}
[\cite{Ambrose:1942fj}] \label{thm:lif_measures}Let ${\cal M}_{\s}(\tau):=\{m\in{\cal M}_{\s}:\ \int_{\TMS}\tau\dd m<\infty\}$
then there exists a bijection 
\[
\begin{array}{cccc}
R: & {\cal M}_{\s}(\tau) & \to & {\cal M}_{\phi}\\
 & m & \mapsto & \frac{m\times{\rm Leb}}{m\times{\rm Leb}(\TMS_{\tau})}
\end{array}
\]
where ${\rm Leb}$ is the Lebesgue measure for the flow direction. 
\end{thm}

In other words, for any continuous function $F:\TMS_{\tau}\to\R$
, we have 
\[
\int_{\TMS_{\tau}}F\dd R(m)=\frac{\int_{\TMS}\Delta_{F}\dd m}{\int_{\TMS}\tau\dd m}.
\]

\begin{thm}
[Equilibrium States for flows; \cite{Iommi:2015th} Theorem 3.4, 3.5 ]\label{thm:equilibrium-states-flow}
Let $F:\TMS_{\tau}\to\R$ be a continuous function such that $\Delta_{F}$
is locally H\"older. Suppose $\Delta_{F}$ has an equilibrium state
$m_{\Delta_{F}}$ such that $\int\tau\dd m_{\Delta_{F}}<\infty$.
Then $F$ has an unique equilibrium state $\mu=R(m_{\Delta_{F}})$. 
\end{thm}

\section{Geodesic Flows for Finite Area Hyperbolic Surfaces\label{sec:Geodesic-Flows-Symbolic}}

\subsection{A Symbolic Model for Geodesics Flows }

In this section, we survey a symbolic model for the geodesic flow.
More precisely, we will construct a geodesic flow invariant subset
$\Omega_{0}$ of the unit tangent bundle, and study it through a symbolic
model. This construction is given by Ledrappier and Sarig in \cite{Ledrappier:2008wq}.
We will mostly follow their notations and use the Poincaré disk model
$\D$ in this section. 

Let $S=S_{g,n}$ be a surface with $g$ genus and $n$ punctures,
$X=X_{\rho}$ be the finite area hyperbolic surface given by the Fuchsian
representation $\rho:\pi_{1}(S)\to\psl(2,\R)$, and $g_{t}:T^{1}X\to T^{1}X$
be the geodesic flow for $X$. In this paper, we only interested in
non-compact surfaces, because the compact cases have been studied
before. In other words, in our discussion $n$ is no less than 1. 
\begin{thm}
[\cite{Tukia:1972ti,Tukia:1973uw}] Suppose $X$ is a non-compact
finite area hyperbolic surface with negative Euler characteristic.
Then there exists a closed ideal hyperbolic polygon $D_{0}\subset\D$
such that \begin{enumerate}[font=\normalfont]

\item The origin is in $D_{0}$. 

\item $D_{0}$ has 2$k$ vertices, and all vertices are on $\vbdy\D$,
where $k=2g+n-1=-\chi(X)+1\geq2$. 

\item These vertices partition $\vbdy\D$ to $2k$ intervals $I_{i}$,
$i\in{\cal {\cal S}}$ where ${\cal S}:=\{1,1',2,2',...,k,k'\}$.
Moreover, each $I_{i}$ can be paired with the other interval $I_{i'}$
such that there exists a pair of M\"obius transformations $g_{i},g_{i'}=g_{i}^{-1}$
with $g_{i}$ maps $I_{i}$ onto $\vbdy\D\backslash I_{i'}$ and $g_{i'}$
maps $I_{s'}$ onto $\vbdy\D\backslash I_{i}$.

\item $X$ is isomorphic to the space obtained by identifying all
pairs of $(I_{i},I_{i'}$) through $g_{i}$ for all $i\in{\cal S}$. 

\item Take $i$ (or $i'$) from each side pair $(I_{i},I_{i'})$
and consider the corresponding M\"obius transformation $g_{i}$,
then 
\[
\G=\rho(\pi_{1}(X))=\langle g_{1},...,g_{k}\rangle
\]
 where $\rho$ is the Fuchsian representation such that $X=\G\backslash\D$. 

\end{enumerate}
\end{thm}

\textcolor{red}{}
\begin{figure}
\textcolor{red}{\includegraphics[scale=0.3]{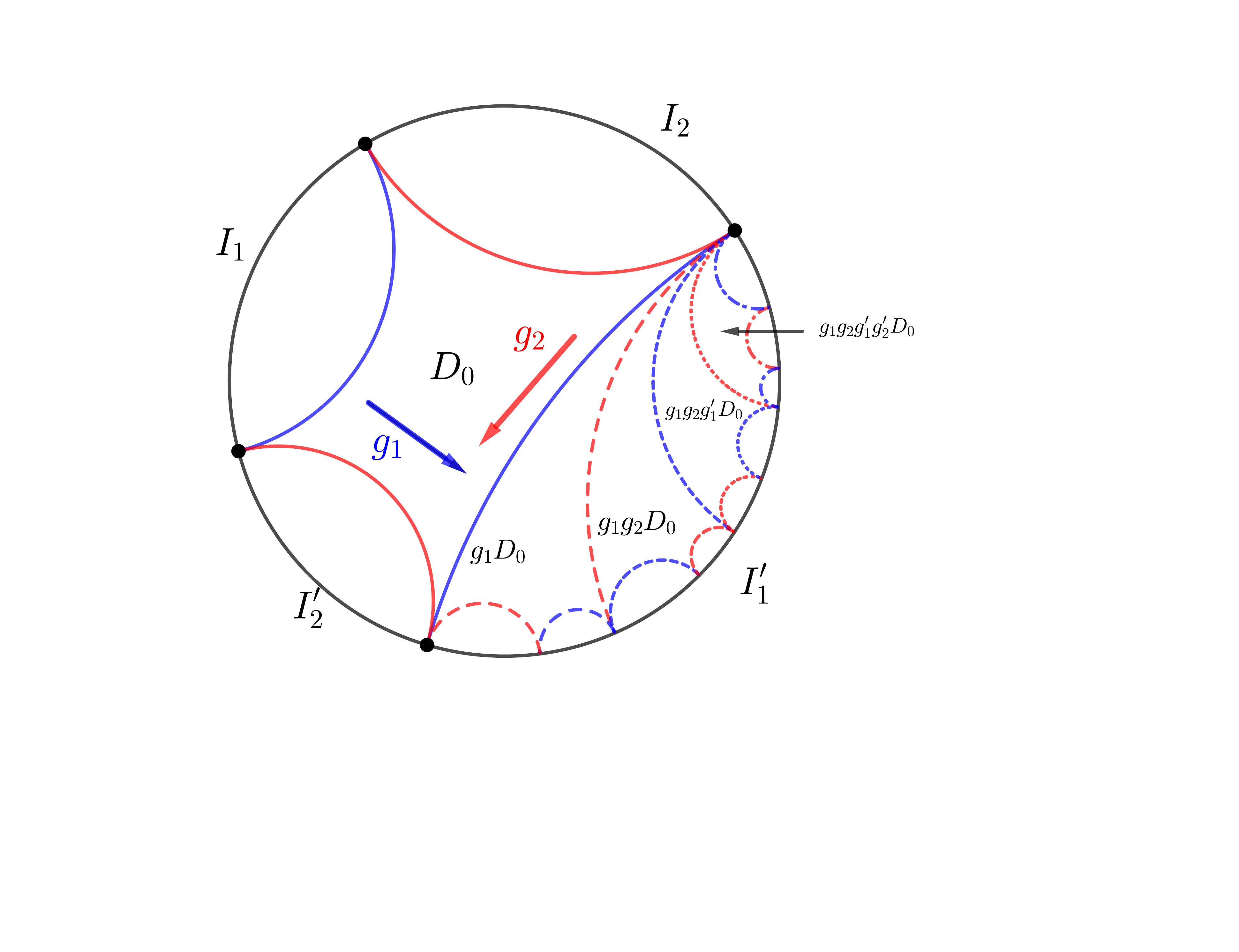}}

\textcolor{red}{\caption{Finite Area Surfaces with Cusps}
\label{Fig:schottky}}
\end{figure}

From now on, for the finite area hyperbolic surface $X$, we use the
generator given in the above theorem, and denote $\G=\langle g_{1},...,g_{k}\rangle$.
Roughly speaking, there are two steps to construct the Ledrappier-Sarig
coding. One first uses the generators $\{g_{1},...,g_{k}\}$ to derive
a Markov shift $(\Sigma_{1},\sigma_{1})$ (i.e., cutting sequences),
then modify $(\Sigma_{1},\sigma_{1})$ to get another Markov shift
$(\Sigma_{A},\sigma_{A})$ on which the first returning map has better
regularity. We will discuss their construction in detail below.

The shape of the fundamental fundamental $D_{0}$ plays a crucial
role in the Ledrappier and Sarig's coding. We start from looking at
vertices of $D_{0}$. Notice that for every vertex $v$ of $D_{0}$,
there exists a (shortest) cycle, say $l$ elements, of edge-pairing
isometries $g_{s_{i}}$ for $1\leq i\leq l$ such that $v$ is the
unique fixed point of $g_{s_{l}}g_{s_{l-1}}...g_{s_{2}}g_{s_{1}}$
provided $g_{s_{i}}...g_{s_{2}}g_{s_{1}}(D_{0})$ and $(g_{s_{1}}^{-1}g_{s_{2}}^{-1}...g_{s_{i}}^{-1})(D_{0})$
touch $\vbdy\D$ at $v$ for all $1\leq i\leq l$. We call $\underline{w}=(s_{1},...,s_{l})$
and $\underline{w'}=(s'_{l},....,s'_{1})$ the \textit{cycles} of
$v$. We denote the set of all vertex cycles by $\mathfrak{C}$, and
$N(\mathcal{\mathfrak{C}})$ is the least common multiplier of length
of cycles of all vertices (see Figure \ref{Fig:schottky}). 

\subsubsection{The Classical Coding}

\textcolor{red}{}
\begin{figure}
\textcolor{red}{\includegraphics[scale=0.35]{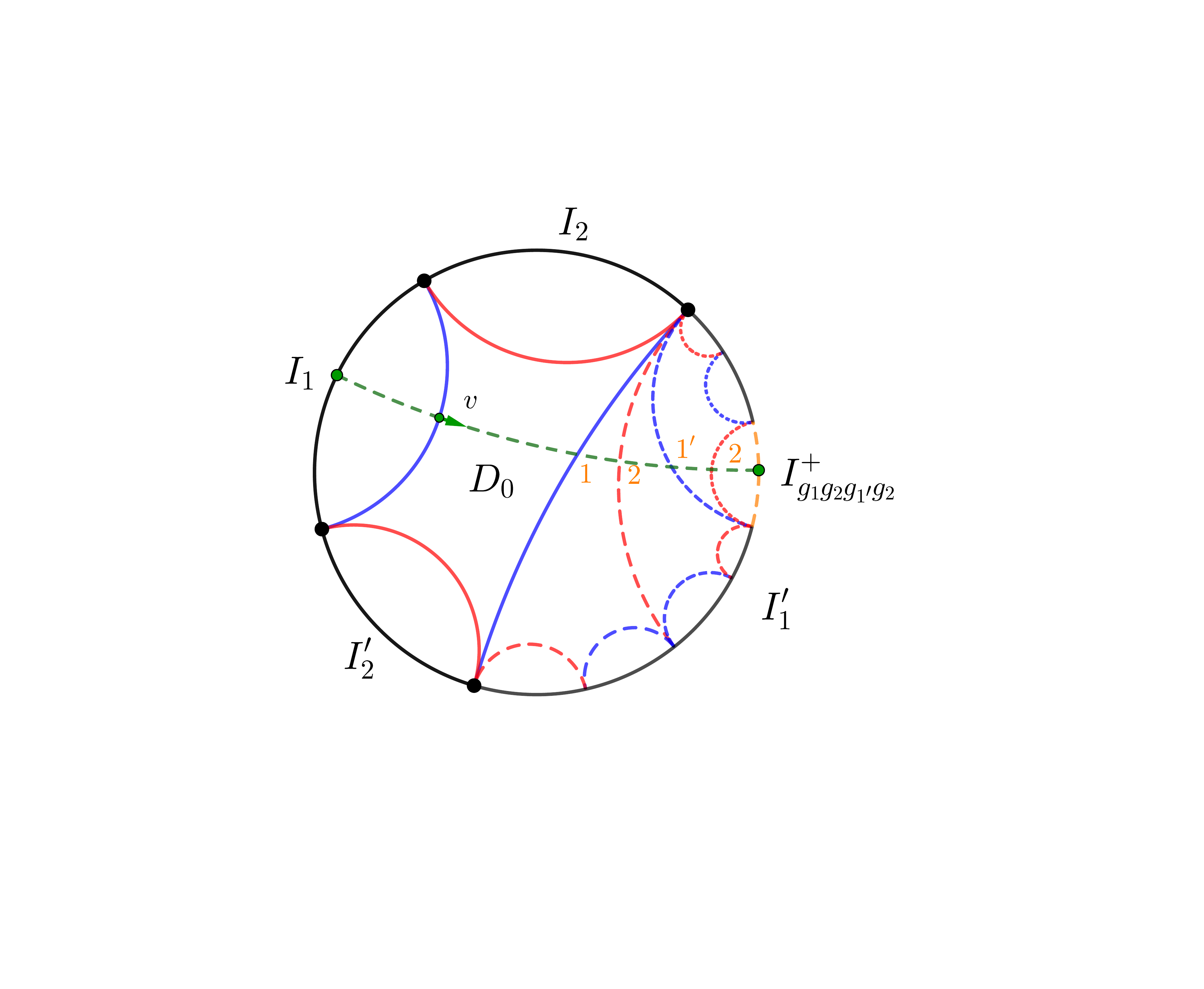}\caption{Classical Coding}
\label{Fig:Classical Coding}}
\end{figure}

Recall that a vector $v\in T^{1}X$ \textit{escapes to infinity} if
$g_{t}(v)$ leaves, eventually, all compact set $K\subset T^{1}M$
as $t\to\infty$ or $-\infty.$ Let $\Omega_{0}\subset T^{1}X$ be
the set of non-escaping vectors. It is clearly that $\Omega_{0}$
is a flow invariant set and contains most of the interesting dynamics. 

A unit vector $v\in T^{1}\D$ based at $\D\cap\partial D_{0}$ is
called \textit{inward pointing} if $g_{t}(v)\in{\rm int}(D_{0})$
for sufficiently small $t$. We denote by $(\partial D_{0})_{{\rm in}}$
the set of all inward pointing vectors. It is not hard to see $(\partial D_{0})_{{\rm in}}$
projects to a Poincar\'e section of $g_{t}:\Omega_{0}\to\Omega_{0}$,
by abusing the notation, we also denote this section by $(\partial D_{0})_{{\rm in}}$. 

In the following, we recall two equivalent methods of coding of geodesic
flows on $\Omega_{0}$: \textit{cutting sequences} and \textit{boundary
expansion}. To derive the coding, we first label edges of $D_{0}$
in the following manner. For each edge $e$ of $D_{0}$, it determines
a boundary interval $I_{s(e)}$ for some $s(e)\in{\cal S}$ such that
$I_{s(e)}$ has the same vertices as $e$ and is on the side of $e$
which does not contain $D_{0}$. We call $s=s(e)\in{\cal S}$ the
\textit{external label} of $e$, and $s'=s'(e)$ the \textit{internal
label} of $e$. See Figure \ref{Fig:Classical Coding} for an illustration. 

Now we are ready to state two canonical coding or Markov partition
associated to $(\partial D_{0})_{{\rm in}}$. For every $v\in(\partial D_{0})_{{\rm in}}$
it is determined by 

\begin{enumerate}[font=\normalfont]

\item \textbf{Cutting sequence} $(x_{k})\in{\cal S}^{\Z}$ : $x_{k}$
are the internal labels of the edges of $D_{0}$ cut by $g_{t}(v)$
where $k=1$ is the first cut in postive time and $k=0$ is the first
cut in non-negative time. 

\item \textbf{Boundary expansion} $(y_{k})\in{\cal S}^{\Z}$: the
lift $\widetilde{(g_{t}v)}\subset T^{1}\D$ is a geodesic on $T^{1}\D$
has a \textit{attracting limit point} (or the ending point) in $\bigcap_{k\geq1}I_{y_{1},...,y_{k}}^{+}$,
and a \textit{repelling limit point} (or the beginning point) in $\bigcap_{k\leq0}I_{y_{0},...,y_{k}}^{-}$
where $I_{s_{1},s_{2},...,s_{k}}^{+}=g_{s_{1}}g_{s_{2}}...g_{s_{k-1}}(I_{s'_{k}})$
and $I_{s_{1},s_{2},...,s_{k}}^{-}=g_{s_{1}}^{-1}g_{s_{2}}^{-1}...g_{s_{k-1}}^{-1}(I_{s_{k}})$.

\end{enumerate}

It is not hard to see $(x_{k})_{k\in\Z}=(y_{k})_{k\in\Z}$ because
all vertices of $D_{0}$ are on $\vbdy\D$. Thus we can and will interchange
in between these two perspectives. In sum, the classical coding means
that for every $v\in(\partial D_{0})_{{\rm in}}$, the geodesic $g_{t}(v)$
correspond to an element in 
\[
\Sigma_{1}:=\{(x_{k})\in\mathcal{S}^{\Z}:x_{k+1}\neq(x_{k})'\}
\]
and $\sigma_{1}$ is the left shift on $\Sigma_{1}$. 

\subsubsection{The Modified Coding}

As pointed out in \cite{Ledrappier:2008wq}, $(\Sigma_{1},\sigma_{1})$
is not ``good'' enough for our purpose. For example, the classcal
coding is not necessarly one to one, and the first returning map is
not regular enough to push the machinery. Thus we need to modify $(\Sigma_{1},\sigma_{1}$)
by looking at a smaller section of the flow $g_{t}:\Omega_{0}\to\Omega_{0}$. 

Fix a number $n^{*}$ large, set $N^{*}=4n^{*}N(\mathfrak{C})$, and
the set of length $N^{*}$ repeating vertiex cycles defined as 
\[
\mathfrak{C}^{*}:=\{\underset{N^{*}/|\underline{w}|\ {\rm copies}}{\underbrace{(\underline{w},\underline{w},...,\underline{w})}}:\ \underline{w}\in\mathfrak{C}\}.
\]
We write $N^{\#}:=\frac{1}{2}N^{*}-1$. Now consider the following
set 
\[
A:=\{y\in\Sigma_{1}:\ \underset{N^{*}}{\underbrace{(y_{-N^{\#},...,}y_{\frac{N^{*}}{2}})}}\notin\mathfrak{C}^{*}\}\subset\Sigma_{1}.
\]
The smaller section $S_{A}\subset(\partial D_{0})_{{\rm in}}$ is
given by 
\[
S_{A}:=\{v\in(\partial D_{0})_{{\rm in}}:\ {\rm the\ cutting\ sequence\ of\ }g_{t}(v)\ {\rm is\ in\ }A\}
\]
 (see Figure \ref{Fig:Smaller Section}).

\textcolor{red}{}
\begin{figure}
\textcolor{red}{\includegraphics[scale=0.25]{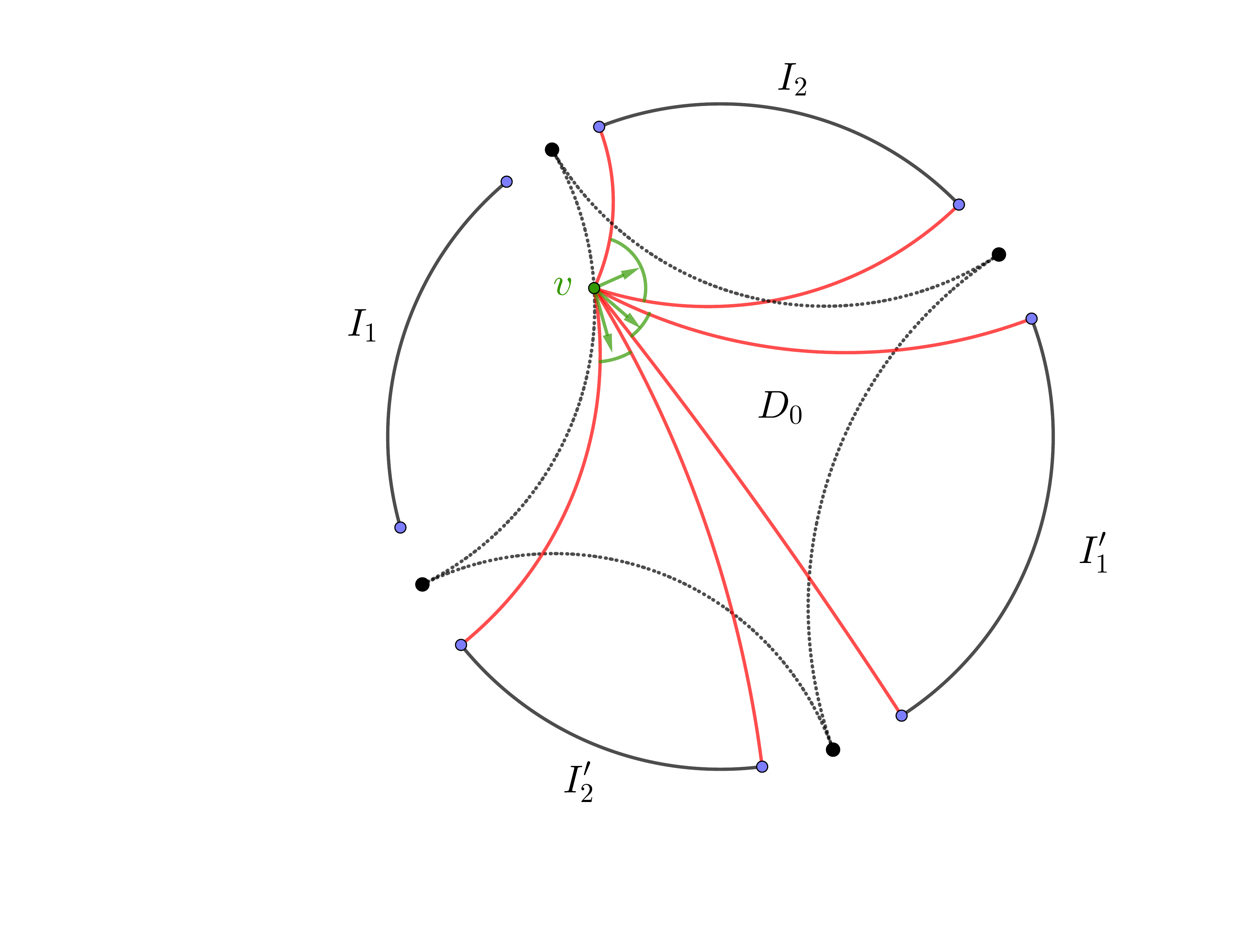}\caption{The Smaller Section}
\label{Fig:Smaller Section}}
\end{figure}
It is not hard to see that $S_{A}$ is a Poincar\'e section of $g_{t}:\Omega_{0}\to\Omega_{0}$.
Moreover, by the combinatorial property of $\mathfrak{C}$ pointed
in \cite[Section 2.1]{Ledrappier:2008wq}, we know for a geodesic
$g_{t}(v)$ with the cutting sequence $(x_{n})_{n\in\Z}$ which stops
returning to $A$ at some point, $(x_{n})$ will eventually repeating
an element $\underline{w}\in\mathfrak{C}^{*}$, i.e., $(x_{n})_{n\in\Z}=(....,x_{n},...,\underline{w},\underline{w},\underline{w},...)$. 

In other words, if $v$ does not escape to infinity, then the cutting
sequence of $g_{t}(v)$ always returns to $A$. More precisely, $\forall x\in A,$
there exists $N=N(x)\in\R$ such that $\sigma_{1}^{N}(x)\in A$. We
define the induced shift map on $A$ by $\sigma_{A}(x):=\sigma_{1}^{N_{A}(x)}$
where $N_{A}(x)=\min\{N\in\N:\ \sigma_{1}^{N}(x)\in A\}.$ 

Now, we are ready to describe a Markov partition of this modified
Markov shift $\sigma_{A}:A\to A$: 

\begin{enumerate}[font=\normalfont]

\item \textbf{Type I}, denoted by $\Sigma_{A}({\rm I})$: cylinders
of length $N^{*}+1$, namely $[x_{-N^{\#}},...,\dot{x}_{0},...,x_{\frac{N}{2}^{*}-1},x_{\frac{N}{2}^{*}}]$,
such that \\ $\underset{{\rm length}=N^{*}}{\underbrace{[x_{-N^{\#}},...,\dot{x}_{0},...,x_{\frac{N}{2}^{*}-1}]}}\subset A$
and $[x_{-N^{\#}+1},...,\dot{x}_{1},...,x_{\frac{N}{2}^{*}}]=\sigma_{1}([x_{-N^{\#}},...,\dot{x}_{0},...,x_{\frac{N}{2}^{*}-1}])\subset A$.
The \textit{shape} of $[\underline{e}]\in\Sigma_{A}({\rm I})$ is
defined as $\mathfrak{s}(\underline{e})=(\underline{e})$.

\item \textbf{Type II}, denoted by $\Sigma_{A}({\rm II})$: cylinders
of length bigger than $N^{*}+1$, namely
\[
B_{l,k}(a,\underline{w},c):=[x_{-N^{^{\#}}},w_{1},...,\dot{w}_{N^{\#}},...,w_{N^{*}},(\underline{w})^{l-1},w_{1},w_{2},...,w_{k},\underline{b}]
\]
 where $a:=x_{-N^{\#}},c\in{\cal S},$ $\underline{w}\in\mathfrak{C}^{*}$,
$l\geq0$, $0\leq k<N^{*}$ are not both zero, and 
\[
\underline{b}:=\begin{cases}
(w_{k+1},...,w_{N^{*}},w_{1},...,w_{k-1},c) & l=0,k\neq0\\
(w_{1},...,w_{N^{*}-1},c) & l\neq0,k=0\\
(w_{k+1},...,w_{N^{*}},w_{1},...,w_{k-1},c) & l,k\neq0
\end{cases}
\]
provided $B_{l,k}(a,\underline{w},c)\subset A$, $[\underline{b}]\subset A$.
The \textit{shape} of $[\underline{e}]\in\Sigma_{A}({\rm II})$ and
of the form $B_{l,k}(a,\underline{w},c)$ is deifned as 
\[
\mathfrak{s}(\underline{e}):=(k,a,\underline{w},c)\in\{0,...,N^{*}-1\}\times{\cal S}\times\mathfrak{C}^{*}\times{\cal S}.
\]

\end{enumerate}

\begin{prop}
[\cite{Ledrappier:2008wq}, Lemma 2.1] $\sigma_{A}:A\to A$ is topologically
mixing, and the Markov partition given by $\Sigma_{A}({\rm I})$ and
$\Sigma_{A}({\rm II})$ has the BIP property. 
\end{prop}

Let $(\Sigma_{A},\sigma_{A})$ be the countable state Markov shift
derived by the Markov partition $\Sigma_{A}({\rm I})$ and $\Sigma_{A}({\rm II})$.
We write the alphabet set of $\TMS_{A}$ by
\[
{\cal S}_{A}:=\{\underline{e}\in\bigcup_{n\geq N^{*}+1}{\cal S}^{n}:\sigma_{1}^{\#}[\underline{e}]\in\Sigma_{A}({\rm I})\cup\Sigma_{A}({\rm II})\}.
\]
Let $\pi_{A}:\Sigma_{A}\mapsto A\subset\Sigma_{1}$ denote the natural
coding map. For $x\in\Sigma_{A}$, we use $x_{0}$ to denote the letter
in the zero-th coordinate. Notice that we can always write $x_{0}=(s_{-N^{\#}},...,s_{n-N^{\#}-1})$
in terms of ${\cal S}$ letters, and in this representation $n-N^{*}$
is the $\sigma_{1}-$return time of $[x_{0}]$. 

\begin{rem}
\ \begin{enumerate}[font=\normalfont]

\item $\Sigma_{A}({\rm I})$ is composed by return time 1 (i.e.,
$N_{A}=1)$ cylinders, and $\underset{n+1\ {\rm terms}}{\underbrace{[x_{-N^{\#},...,}x_{n-N^{\#}}]}}\in\Sigma_{A}({\rm II})$
has return time $n-N^{*}$. 

\item There are only finitely different shapes $\mathfrak{s}(\underline{a})$
for all $\underline{a}\in{\cal S}_{A}$. 

\item The length $|\underline{a}|$ for $\underline{a}\in{\cal S}_{A}$
is unbounded. 

\end{enumerate}
\end{rem}

Recall that every $x=(x_{k})\in\Sigma_{A}$ determines a point $\pi_{A}(x)=(s_{i})\in\Sigma_{1}$,
and $\pi_{A}(x)$ corresponds to a unit tangent vector $v(x)\in S_{A}\subset(\partial D_{0})_{{\rm in}}$.
We write $\xi(x)$ the attracting limit point of $v(x)$ and $\eta(x)$
the repelling fixed point of $v(x)$. Since $\xi(x)=\bigcap_{k\geq1}I_{s_{1},...,s_{k}}^{+}$and
$\eta(x)=\bigcup_{k\leq0}I_{s_{0},...,s_{k}}^{-}$ where $\pi_{A}(x)=(s_{i})_{i\in\Z}$,
we know that $\xi(x)$ only depends on $x^{+}=(x_{k})_{k\geq0}$ and
$\eta(x)$ only depends on $x^{-}=(x_{k})_{k\leq0}$ . 

\begin{defn}
The \textit{geometric potential} $\tau:\Sigma_{A}\to\R$ is defined
as 
\[
\tau(x):=B_{\xi(x)}(o,(g_{x_{0}})o)
\]
 where $o$ is the origin, $x_{0}=(s_{-N^{\#}},...,s_{n-N^{\#}-1})\in{\cal S}_{A}$,
and $g_{x_{0}}=g_{s_{1}}\circ...\circ g_{s_{n-N^{*}}}$.
\end{defn}

\begin{prop}
[Geometric Potential (I), \cite{Ledrappier:2008wq}, Lemma 2.2] Let
$(\Sigma_{A},\sigma_{A})$ be the Markov shift constructed above.
Then
\end{prop}

\begin{enumerate}[font=\normalfont]

\item Suppose $v$ generates a closed geodesics, namely $g_{l(v)}v=v,$
then there exists a unique (up to permutations) $x=(\overline{x_{1}x_{2}...x_{m}})\in{\rm Fix}^{m}(\Sigma_{A})$
such that $l(v)=S_{m}\tau(x)$, and vice versa.

\item $\tau$ is locally H\"older continuous.

\item $\tau$ only depends on the future coordinates, i.e., if $x_{0}^{\infty}=y_{0}^{\infty}$
then $\tau(x)=\tau(y)$.

\item $\exists C,K>0$ such that $\tau(x)+\tau(\sigma(x))+...+\tau(\sigma^{n}(x))\geq C$
for all $n\geq K$.

\end{enumerate}

Since the geometric potential $\tau$ is only dependent on the future
coordinate, we can focus on $(\TMS_{A},\sigma_{A})$ the one-sided
countable Markov shift deduced from $(\Sigma_{A},\sigma_{A})$ by
forgetting the past coordinate.

\begin{prop}
[Geometric Potential (II), \cite{Ledrappier:2008wq}, Lemma 3.1]\label{prop:Geo_Poten_II}
On the one-sided countable Markov shift $(\TMS_{A},\sigma_{A})$,
we have
\end{prop}

\begin{enumerate}[font=\normalfont]

\item $\tau$ has a unique equilibrium state $m_{-\tau}$ and $\int_{\TMS_{A}}\tau\dd m_{-\tau}<\infty$.

\item The Liouville measure $m_{L}$ on $T^{1}X$ is given by $m_{L}=R(m_{-\tau})\circ\widetilde{\pi}_{A}^{-1}$
where $R:{\cal {\cal M}_{\sigma}\to{\cal M}_{\tau}}$ given in Theorem
\ref{thm:lif_measures}.

\item $P(-\tau)=0$.

\item \label{Prop:Geo_Poten_II_tau_bdd}$\tau$ is bounded on $\Sigma_{A}({\rm I})$,
and there exists $C_{1}>0$ such that $2\ln|x_{0}|-C_{1}\leq\tau(x)\leq2\ln|x_{0}|+C_{1}$
for all $x\in\Sigma_{A}({\rm II})$. 

\end{enumerate}

\begin{proof}
Everything is in \cite[Lemma 3.1]{Ledrappier:2008wq}, and only the
first assertion of (\ref{Prop:Geo_Poten_II_tau_bdd}) needs more explorations.
Let $[x_{0}]\in\Sigma_{A}({\rm I)}$ and $x\in[x_{0}]$. We can write
$x_{0}=(s_{-N^{\#}},...,s_{\frac{N^{*}}{2}-1},s_{\frac{N^{*}}{2}})$,
$g_{x_{0}}(x)=g_{s_{1}}$, and $\tau(x)=B_{\xi(x)}(o,g_{s_{1}}o)$
where $s_{i}\in{\cal S}$ for $i=-N^{\#},...,\frac{N^{*}}{2}$. Recall
that in the disc model, $B_{\xi}(o,y)=\ln\frac{1-|y|^{2}}{|\xi-y|^{2}}$.
Since $\xi(x)\in I_{s_{1},s_{2},...,s_{\frac{N^{*}}{2}}}^{+}$, it
is not hard to see $B_{\xi(x)}(o,g_{s_{1}}o)=\ln(\frac{1-|g_{s_{1}}o|^{2}}{|\xi(x)-g_{s_{1}}o|})$
is (uniformly) bounded for all $x\in[x_{0}]$. Notice that this bound
depends on $[x_{0}]\in\Sigma_{A}({\rm I}).$ We can find a universal
bound $\tau(x)$ on $\Sigma_{A}({\rm I})$ because $|\Sigma_{A}({\rm I})|<\infty$. 
\end{proof}

\begin{rem}
\

\begin{enumerate}[font=\normalfont]

\item By standard techniques in symbolic dynamics, we know $\tau$
is cohomologus to $\tau'$ which is locally H\"older and bounded
away from zero (cf. \cite[Lemma 3.8]{Kao:2018th}). From now on, we
will use $\tau'$ to replace $\tau$ whenever $\tau$ needs to be
bounded away from zero. Abusing the notation, we will continue denote
$\tau'$ by $\tau$. 

\item In \cite{Ledrappier:2008wq}, the constant $C_{1}$ in Proposition
\ref{prop:Geo_Poten_II} (\ref{Prop:Geo_Poten_II_tau_bdd}) depends
on the shape of $x_{0}$. Because there are only finitely many shapes,
we can replace it by a universal constant. 

\end{enumerate}
\end{rem}

\subsection{Type-preserving Finite Area Fuchsian Representations}

In the this subsection, we consider $\rho_{1},\rho_{2}$ two type-preserving
finite area Fuchsian representations. The Fenchel-Neilsen Isomorphism
Theorem (cf. Theorem \ref{thm:Fenchel-Neilsen-Thm}) shows that there
exists a bilipschitz map taking the limit set $\L(\rho_{1}(\pi_{1}S))$
and fundamental domain of $X_{\rho_{1}}$ to $\L(\rho_{2}(\pi_{1}S))$
and the fundamental domain of $X_{\rho_{2}}$, and hence $\Lambda_{0}(\rho_{1})$
to $\L_{0}(\rho_{2})$. Hence, the suspension flows corresponding
the geodesic flows on $\Omega_{0}(\rho_{1})$ and $\Omega_{0}(\rho_{2})$
correspond to the same Markov shift $(\TMS_{A},\s_{A})$ but different
roof functions $\tau_{\rho_{1}}$, $\tau_{\rho_{2}}$, respectively.
The following result shows that we have a nice control of these roof
functions. 

\begin{cor}
\label{cor:bdd-diff-loof-fcn}There exists $C>0$ such that $|\tau_{\rho_{1}}(x)-\tau_{\rho_{2}}(x)|<C$
for all $\TMS_{A}$.
\end{cor}

\begin{proof}
It follows immediately from Proposition \ref{prop:Geo_Poten_II} (\ref{Prop:Geo_Poten_II_tau_bdd}).
\end{proof}

In the second part of this subsection, we discuss the intersection
number ${\rm I}(\rho_{1},\rho_{2})$ of $\rho_{1}$ and $\rho_{2}$
proposed by Thurston. Recall that ${\rm I}(\rho_{1},\rho_{2})$ of
$\rho_{1}$ and $\rho_{2}$ is defined as 
\[
{\rm I}(\rho_{1},\rho_{2}):=\lim_{n\to\infty}\frac{l_{2}[\g_{n}]}{l_{1}[\g_{n}]}
\]
 where $\{[\g_{n}]\}_{n=1}^{\infty}$ is a sequence of conjugacy classes
for which the associated closed geodesics $\g_{n}$ become equidistributed
on $\rho_{1}(\pi_{1}S)\backslash\H$ with respect to the Liouville
measure. However, it is unclear why ${\rm I}(\rho_{1},\rho_{2})$
is well-defined, especially, when $S$ has punctures. We will discuss
this issue in Proposition \ref{prop:intersection-dynamics} where
we give ${\rm I}(\rho_{1},\rho_{2})$ a dynamics characterization. 

We now can state and prove the main result of this subsection: characterizing
${\rm I}(\rho_{1},\rho_{2})$ by the symbolic model.

\begin{prop}
\label{prop:intersection-dynamics} Suppose $\rho_{1},\rho_{2}$ are
two type-preserving finite area Fuchsian representations. Then the
intersection ${\rm I}(\rho_{1},\rho_{2})$ is well-defined. Moreover,
if $\tau$, $\kappa$ are the geometric potential for $\rho_{1}$,
$\rho_{2}$, respectively, then 
\[
{\rm I(\rho_{1},\rho_{2})=}\frac{\int\kappa\dd m_{-\tau}}{\int\tau\dd m_{-\tau}}
\]
where $m_{-\tau}$ is the equilibrium state of $\tau$. 
\end{prop}

\begin{proof}
Since the lifted measure $R(m_{-\tau})$ on $\Sigma_{\tau}^{+}$ is
the unique measure of maximum entropy, we know it ergodic. Plus $(\TMS_{A},\sigma_{A})$
is topologically mixing, we know $R(m_{-\tau})$ is the weak-start
limit of some period orbits $\{\lambda_{n}\}\subset\TMS_{\tau}$,
that is, $\frac{\delta_{\lambda_{n}}}{l(\lambda_{n})}\overset{weak^{*}}{\longrightarrow}R(m_{-\tau})$
where\textcolor{red}{{} }$\lambda_{n}=\{(x_{n},t)\in\TMS_{\tau}:\ 0\leq t\leq S_{m}\tau(x_{n}),\ x_{n}\in{\rm Fix}^{m}\ {\rm for}\ {\rm some\ }m\}$,
$\delta_{\lambda_{n}}$ is the Dirac measure on $\lambda_{n}$, and
$l(\lambda_{n}):=S_{m}\tau(x_{n})=l_{1}[\g_{n}]$ with $\g_{n}\in\pi_{1}S$
the hyperbolic elements corresponding to $\lambda_{n}$. 

Let us consider the (symbolic) reparametrization function $\psi(x,t)=\frac{\kappa(x)}{\tau(x)}$
for $0\leq t\leq\tau(x)$ .\textcolor{red}{{} }By Proposition \ref{prop:Geo_Poten_II}
and Corollary \ref{cor:bdd-diff-loof-fcn}, we know $\psi$ is a bounded
locally H\"{o}lder continuous function such that $\int_{\rho_{1}(\g_{n})}\psi(\phi_{t})\dd t=l_{2}[\g_{n}]$.
Thus , 
\[
\left.\frac{\langle\delta_{o_{n}},\psi\rangle}{l(o_{n})}\right|_{\TMS_{\tau}}=\frac{l_{2}[\g_{n}]}{l_{1}[\g_{n}]}\to\int_{\TMS_{\tau}}\psi\dd R(m_{-\tau})=\frac{\int_{\TMS}\kappa\dd m_{-\tau}}{\int_{\TMS}\tau\dd m_{-\tau}}
\]
where the last equality follows Theorem \ref{thm:lif_measures}.
\end{proof}

\section{Phase Transitions for Geodesic Flows\label{sec:Phase-Transitions}}

Through out this section, let $\rho_{1}$ and $\rho_{2}$ be two type-preserving
finite volume Fuchsian representations, and we write $X_{1}=\G_{1}\backslash\mathbb{D}$
and $X_{2}=\G_{2}\backslash\mathbb{D}$ where $\G_{1}=\rho_{1}(\pi_{1}(S))$
for $i=1,2$. Following the above section, let $(\TMS,\s)=(\TMS_{A},\s_{A})$
be the Markov shift associated with $X_{1}$ and $X_{2}$, and we
denote their geometric potentials by $\tau$ and $\kappa$, respectively. 

To derive the analyticity of pressure, we need to locate the place
where phase transition happens. As in \cite{Kao:2018th}, we have
the following observation. 

\begin{thm}
[Phase Transition] \label{thm:phase-transition}Suppose $a,b\geq0$,
$a+b\neq0$, and $\tau,\kappa$ are given above. Then we have 
\[
P(-t(a\tau+b\kappa))=\begin{cases}
{\rm analytic} & t>\frac{1}{2(a+b)},\\
\infty & t<\frac{1}{2(a+b)}.
\end{cases}
\]
Moreover, there exists a unique $t_{ab}\in(\frac{1}{2(a+b)},\infty)$
such that $P(-t_{a,b}(a\tau+b\kappa))=0$.
\end{thm}

\begin{proof}
By Theorem \ref{thm:analyticity-pressure}, we know it is sufficient
to show 
\[
P(-t(a\tau+b\kappa))=\begin{cases}
{\rm finite} & t>\frac{1}{2(a+b)},\\
\infty & t<\frac{1}{2(a+b)}.
\end{cases}
\]

Recall \cite[Theorem 2.19]{Mauldin:2003dn}, we know for any locally
H\"older continuous function $f$, $P(f)<\infty$ if and only if
\[
Z_{1}(f):=\sum_{x_{0}\in{\cal S}_{A}}e^{\sup\{f(x):x\in[x_{0}]\}}<\infty.
\]
By Proposition \ref{prop:Geo_Poten_II}, there exists constants $C_{1}$,
$C_{2}>0$ such that
\begin{alignat*}{1}
Z_{1}(-t(a\tau+b\kappa))= & \sum_{x_{0}\in\Sigma_{A}({\rm I})}e^{\sup\{-t(a\tau+b\kappa):x\in[x_{0}]\}}+\sum_{x_{0}\in\Sigma_{A}({\rm II})}e^{\sup\{-t(a\tau+b\kappa):x\in[x_{0}]\}}\\
\leq & C_{1}+\sum_{x_{0}\in\Sigma_{A}({\rm II})}\sum_{|x_{0}|=N^{*}+1}^{\infty}e^{-2t(a\tau+b\kappa))\log|x_{0}|+C}\\
= & C_{1}+C_{2}\sum_{x_{0}\in\Sigma_{A}({\rm II})}\sum_{|x_{0}|=N^{*}+1}^{\infty}e^{-2t(a\tau+b\kappa))\log|x_{0}|}.
\end{alignat*}
Similarly, there exists constants $C_{3}$, $C_{4}>0$ such that 
\[
Z_{1}(-t(a\tau+b\kappa))\geq C_{3}+C_{4}\sum_{x_{0}\in\Sigma_{A}({\rm II})}\sum_{|x_{0}|=N^{*}+1}^{\infty}e^{-2t(a\tau+b\kappa))\log|x_{0}|}.
\]
Thus, it is clear that $Z_{1}(-t(a\tau+b\kappa))<\infty$ if and only
if $t>\frac{1}{2(a+b)}$. 

Lastly, fix $a,b$ with $a,b\geq0,a+b\neq0$, then the computation
in \cite[Theorem 2.19]{Mauldin:2003dn} showed that, in our case,
$Z_{1}(-t(a\tau+b\kappa))\to\infty$ as $t\to\frac{1}{2(a+b)}$ implies
$P(-t(a\tau+b\kappa))\to\infty$ as $t\to\frac{1}{2(a+b)}$. In particular,
taking $t$ close to $\frac{1}{2(a+b)}$, we have $P(-t(a\tau+b\kappa))>0$.
Moreover, it is obvious that $P(-t(a\tau+b\kappa))<0$ when $t$ is
big enough. Hence, by the analyticity and the monotonicity of the
pressure, we know there exists a unique $t_{a,b}$ such that $P(-t_{a,b}(a\tau+b\kappa))=0$. 
\end{proof}

\begin{cor}
\label{cor:sol_analytic}The set $\{(a,b):\ a,b\geq0,a+b\neq0,{\rm \ and}\ P(-a\tau-b\kappa)=0\}$
is a real analytic curve. 
\end{cor}

\begin{proof}
The proof of \cite[Theorem 3.14]{Kao:2018th} applies to here. In
short, by Theorem \ref{thm:phase-transition}, it makes sense to discuss
solutions to $P(-a\tau-b\kappa)=0$. To see the solution set is a
real analytic curve one only needs to apply the Implicit Function
Theorem, because we know 
\[
\left.\partial_{b}P(-a\tau-b\kappa)\right|_{(a_{0},b_{0})}=-\int\kappa\dd\nu_{-a_{0}\tau-b_{0}\kappa}<-c
\]
where $c>0$ is a lower bound for $\kappa$ and $\nu_{-a_{0}\tau-b_{0}\kappa}$
is the Gibbs measure for $-a_{0}\tau-b_{0}\kappa$.
\end{proof}

\section{Manhattan Curves, Critical Exponents, and Rigidity\label{sec:Manhattan-and-Rigidity}}

In this section, we will prove Theorem \ref{Thm:manhattan-strictly-convex}
and Theorem \ref{thm-entropy-rigidity}. The ideas most follow \cite{Kao:2018th}.
In \cite{Kao:2018th}, the author used results of Paulin, Pollicott
and Schapira \cite{Pollicott:2012ud} to analyze the geometric Gurevich
pressure over the geodesics flow. The general frame work in \cite{Pollicott:2012ud}
includes finite area hyperbolic surfaces. Nevertheless, for the completeness,
we will give outlines of the proofs, and reader can find all details
in \cite{Kao:2018th}.

Following the notations in Section \ref{sec:Phase-Transitions}, let
$\rho_{1},\rho_{2}$ be two type-preserving finite area Fuchsian representations,
$X_{1}=X_{\rho_{1}}$ and $X_{2}=X_{\rho_{2}}$ be the corresponding
hyperbolic surfaces, and $\tau$,$\kappa$ be the corresponding geometric
potentials over the Markov shift $(\TMS,\sigma)=(\TMS_{A},\s_{A})$.\textcolor{red}{{}
}Recall that the Poincaré series $Q_{\rho_{1},\rho_{2}}^{a,b}(s)$
of the weighted Manhattan metric $d_{\rho_{1},\rho_{2}}^{a,b}$ is
defined by 
\[
Q_{\rho_{1},\rho_{2}}^{a,b}(s):=\sum_{\g\in\pi_{1}(S)}\exp(-s\cdot d_{\rho_{1},\rho_{2}}^{a,b}(o,\g o)),
\]
 $\delta_{\rho_{1},\rho_{2}}^{a,b}$ is the critical exponent of $Q_{\rho_{1},\rho_{2}}^{a,b}$,
and the Manhattan curved ${\cal C}(\rho_{1},\rho_{2})$ of $\rho_{1}$
and $\rho_{2}$ is the set $\{(a,b)\in\R_{\geq0}\times\R_{\geq0}\backslash(0,0):\ {\displaystyle \delta_{\rho_{1},\rho_{2}}^{a,b}=1\}}$.
For the brevity, we will drop the subscript $\rho_{1},\rho_{2}$ in
the rest of this subsection. 

The goal of this subsection is to prove the following theorem:

\begin{thm}
\cite[Section 4]{Kao:2018th} \label{prop:pressue_matchs} \label{cor:Bowen's_formula}
Suppose $a,b\geq0$ and $ab\neq0$, then 
\[
P(-\delta_{\rho_{1},\rho_{2}}^{a,b}(a\tau+b\kappa))=0.
\]
In particular, $(a,b)\in{\cal C}(\rho_{1},\rho_{2})$ if and only
if the pair $(a,b)$ satisfies $P(-a\tau-b\kappa)=0$. 
\end{thm}

\begin{proof}
As we mentioned before results in \cite[Section 4]{Kao:2018th} are
applicable here. Here we give a brief outline of the proof. We consider
following growth rates and their relations:

\begin{itemize}[font=\normalfont]

\item the geometric Gurevich pressure $P_{{\rm Geo}}^{a,b}$ given
by growth rates of closed orbits on $T^{1}X_{1}$: 
\[
P_{{\rm Geo}}^{a,b}:={\displaystyle \limsup_{s\to\infty}\frac{1}{s}\log Z_{W}(s)}
\]
where $Z_{W}(s):={\displaystyle \sum_{\stackrel[\lambda\in{\rm Per}_{1}(s)]{\lambda\cap W\neq\phi}{}}e^{-al_{1}[\lambda]-bl_{2}[\lambda]}}$
where $W\subset T^{1}X_{1}$ is a relatively compact open set and
${\rm Per}_{1}(s):=\{\lambda:\ \mbox{\ensuremath{\lambda\ }is a closed orbit on }T^{1}X_{1}{\rm \ and\ }l_{1}[\lambda]\leq s\}$.

\item the critical exponent $\overline{\delta}^{a,b}$ proposed in
\cite{Pollicott:2012ud}: $\overline{\delta}^{a,b}$ is the critical
exponent of $Q_{{\rm PPS},x,y}^{a,b}(s):={\displaystyle \sum_{\g\in\pi_{1}(S)}e^{-d^{a,b}(x,\g y)-sd(x,\rho_{1}(\g)y)}}$
the \textit{Paulin-Pollicott-Schapira's (PPS) Poincaré series.}

\item Let $\psi(x,t):=\frac{\kappa(x)}{\tau(x)}:\TMS_{\tau}\to\R$
for $t\in[0,\tau(x)]$. Then \cite[Lemma 4.7]{Kao:2018th} showed
that $P_{\phi}(-a-b\psi)=0\iff P(-a\tau-b\kappa)=0$. 

\item \cite[Lemma 4.3, 4.4]{Kao:2018th} showed that $P_{{\rm Geo}}^{a,b}=\overline{\delta}^{a,b}=P_{\phi}(-a-b\psi).$

\item \cite[Lemma  4.5]{Kao:2018th} pointed out that $\overline{\delta}^{a,b}=0\iff\delta^{a,b}=1$.

\end{itemize}

In sum, we have 
\begin{alignat*}{1}
\delta^{a,b}=1 & \iff\overline{\delta}^{a,b}=0\\
 & \iff P_{{\rm Geo}}^{a,b}=0\\
 & \iff P_{\phi}(-a-b\psi)=0\\
 & \iff P(-a\tau-b\kappa)=0.
\end{alignat*}
Thus, $P(-t_{a,b}(a\tau+b\kappa))=0\iff\delta^{t_{a,b}a,t_{a,b}b}=1$
, i.e. $Q_{{\rm PPS},o,o}^{t_{a,b}a,t_{a,b}b}(s)={\displaystyle \sum_{\g\in\pi_{1}(S)}e^{-t_{ab}d^{a,b}(o,\g o)}}$
has critical exponent 1. Hence, $Q_{{\rm PPS},o,o}^{t_{a,b}a,t_{a,b}b}(s)={\displaystyle \sum_{\gamma\in\pi_{1}(S)}e^{-sd^{a,b}(o,\g o)}}$
has critical exponent $t_{a,b}$, and thus, $\delta^{a,b}=t_{a,b}$.
\end{proof}

By Corollary \ref{cor:sol_analytic} and the above theorem, we have: 

\begin{cor}
\label{cor:The-Manhattan-curve-pressure}The Manhattan curve ${\cal C}(\rho_{1},\rho_{2})$
is a real analytic curve give by, for $a,b\geq0$ and $a+b\neq0$,
\[
{\cal C}(\rho_{1},\rho_{2})=\{(a,b):\ P(-a\tau-b\kappa)=0\}.
\]
\end{cor}

The following theorem is Bowen's formula which characterize the topological
entropy of the geodesic flow by the pressure and the geometric potential. 

\begin{cor}
\label{cor:Bowen-formula-tau}Suppose $\rho_{1}$ is a finite volume
Fuchsian representation $\rho_{1}$. Then
\[
P(-1\cdot\tau)=0
\]
 where $1$ is the critical exponent of $\rho_{1}(\pi_{1}(S))$.
\end{cor}

\begin{proof}
It follows Proposition \ref{prop:Geo_Poten_II} and the fact that
when $\rho_{1}$ is a finite area Fuchsian representation then the
critical exponent of $\rho_{1}(\pi_{1}(S))$ is 1 (cf. \cite{Otal:2004fn}). 
\end{proof}

Notice that by Bowen's formula and the Implicit Function Theorem,
we can prove that the pressure varies analytically when $\tau$ varies
analytically with $P(-\tau)=0$. 

Now we are ready to prove Theorem \ref{Thm:manhattan-strictly-convex}.

\begin{thm}
[Theorem \ref{Thm:manhattan-strictly-convex}]\label{thm:man-strctily-conv}
The ${\cal C}(\rho_{1},\rho_{2})$ is a convex real analytic curve.
Moreover, ${\cal C}(\rho_{1},\rho_{2})$ is strictly convex unless
$\rho_{1}$ and $\rho_{2}$ are conjugate in ${\rm PSL}(2,\R)$, in
such cases ${\cal C}(\rho_{1},\rho_{2})$ is a straight line.
\end{thm}

\begin{proof}
The analyticity of ${\cal C}$ is proved in Corollary \ref{cor:The-Manhattan-curve-pressure}.
To show the remaining parts, we first notice that by H\"older's inequality
the Manhattan curve ${\cal C}$ is always convex, and because ${\cal C}$
is real analytic we know ${\cal C}$ is either a straight line or
strictly convex. It is clear that if $\rho_{1}$ and $\rho_{2}$ are
conjugate then ${\cal C}$ is a straight line. We claim that if ${\cal C}$
is a straight line then $\rho_{1}$ and $\rho_{2}$ are conjugate
in ${\rm PSL}(2,\R)$.

To see this, suppose ${\cal C}$ is a straight line. Then the slope
of this line is $-1$ because $(1,0)$,$(0,1)\in{\cal C}$. In other
words, we have 
\begin{equation}
-1=-\frac{\int\tau\dd m_{-\tau}}{\int\kappa\dd m_{-\tau}}=-\frac{\int\tau\dd m_{-\kappa}}{\int\kappa\dd m_{-\kappa}}\label{eq:rigidity}
\end{equation}
 where $m_{-\tau}$, $m_{-\kappa}$ are the equilibrium states for
$-\tau$ and $-\kappa$, respectively. 

It is sufficient to show that $\tau$ and $\kappa$ are cohomologus,
because $\tau\sim\kappa$ implies that $X_{1}$ and $X_{2}$ has the
same marked length spectrum, and which implies that $\rho_{1}$ and
$\rho_{2}$ are conjugate in $\psl(2,\R)$ (cf. Theorem \ref{thm:proportinal marked length spectrum}). 

To see $\tau$ and $\kappa$ are cohomologus, by Theorem \ref{thm:props-of-Gibbs-measures},
it is enough to show $m_{-\tau}=m_{-\kappa}$. Notice that $m_{-\tau}$,
$m_{-\kappa}$ are the equilibrium state of $-\tau$, $-\kappa$,
respectively, we have 
\[
h_{\s}(m_{-\tau})-\int\tau\dd m_{-\tau}=P(-\tau)=0=P(-\kappa)=h_{\s}(m_{-\kappa})-\int\kappa\dd m_{-\kappa}.
\]
 Moreover, by equation $($\ref{eq:rigidity}$)$, we know $\int\kappa\dd m_{-\tau}=\int\tau\dd m_{-\tau}$.
Thus, we get 
\[
h_{\s}(m_{-\tau})+\int(-\kappa)\dd m_{-\tau}=0=P(-\kappa).
\]
In other words, $m_{-\tau}$ is a equilibrium state for $-\kappa$,
and by the uniqueness of equilibrium states, we get $m_{-\kappa}=m_{-\tau}$. 
\end{proof}

Using the strictly convexity of the Manhattan curve, we have the following
rigidity results. 

\begin{thm}
[Bishop-Steiger Rigidity; Theorem \ref{thm-entropy-rigidity}] Suppose
$\rho_{1},\rho_{2}$ is two type-preserving finite volume Fuchsian
representations. Then $h_{{\rm BS}}(\rho_{1},\rho_{2})\leq\frac{1}{2}$.
Moreover, the equality holds if and only if $\rho_{1}$ and $\rho_{2}$
are conjugate in $\psltr$.
\end{thm}

\begin{proof}
We first notice that it is a standard and well-known procedure (cf.
\cite[Theorem 4.8]{Kao:2018th}) to show
\[
\delta^{1,1}=h_{{\rm BS}}=\lim_{T\to\infty}\frac{1}{T}\ln\left(\#\{[\g]\in[\pi_{1}(S)]:d(o,\rho_{1}(o,\g o)+d(o,\rho_{2}(o,\g o)\leq T\}\right).
\]
Moreover, since $(\frac{1}{2},\frac{1}{2})$ is the middle point of
$(0,1),(1,0)\in{\cal C}(\rho_{1},\rho_{2})$, by Theorem \ref{thm:man-strctily-conv},
we know $(\frac{1}{2},\frac{1}{2})$ is above $\delta^{1,1}\cdot(1,1)\in{\cal C}(\rho_{1},\rho_{2})$
and $\delta^{1,1}=\frac{1}{2}$ if and only if ${\cal C}$ is a straight
line. 
\end{proof}

\begin{thm}
[Thurston's Rigidity; Theorem \ref{thm-entropy-rigidity}] Suppose
$\rho_{1},\rho_{2}$ is two type-preserving finite volume Fuchsian
representations. Then 
\[
{\rm I}(\rho_{1},\rho_{2})\geq1
\]
 and equals to 1 if and only if $\rho_{1}$ and $\rho_{2}$ are conjugate
in $\psltr$.
\end{thm}

\begin{proof}
By Proposition \ref{prop:intersection-dynamics}, we know 
\[
{\rm I}(\rho_{1},\rho_{2})=\frac{\int\kappa\dd m_{-\tau}}{\int\tau\dd m_{-\tau}}.
\]

Recall that $m_{-\kappa}$ is the equilibrium state of $-\kappa$
and $m_{-\tau}\in{\cal M}_{\s}$ we have 
\[
0=P(-\kappa)=h_{\s}(m_{-\kappa})-\int\kappa\dd m_{-\kappa}\geq h_{\s}(m_{-\tau})-\int\kappa\dd m_{-\tau}.
\]
Notice that $m_{-\tau}$ is the equilibrium state of $-\tau$, i.e.,
$0=P(-\tau)=h_{\s}(m_{-\tau})-\int\tau\dd m_{-\tau}$, we have 
\[
\int\tau\dd m_{-\tau}=h_{\sigma}(m_{-\tau})\leq\int\kappa\dd m_{-\tau}.
\]
The rigidity part has proved in Theorem \ref{thm:man-strctily-conv}.
More precisely, we proved in Theorem \ref{thm:man-strctily-conv}
that if $1=\frac{\int\kappa\dd m_{-\tau}}{\int\tau\dd m_{-\tau}}$
then $\rho_{1}\sim\rho_{2}$ in $\psl(2,\R)$. 
\end{proof}

\section{The Pressure Metric \label{sec:The-Pressure-Metric}}

\subsection{The Pressure Metric and Thurston's Riemannian Metric}

The aim of this subsection is to construct a Riemannian metric for
the Teichm\"uller space of surfaces with punctures. Using the symbolic
model of geodesics flows discussed in Section \ref{sec:Geodesic-Flows-Symbolic},
we can relate the Teichm\"uller space with the space of geometric
potentials. 

Recall that $S=S_{g,n}$ is an orientable surface of $g$ genus and
$n$ punctures and with negative Euler characteristic. The Teichm\"uller
space ${\cal T}(S)$ is the space of conjugacy classes of finite area
type-preserving Fuchsian representations. By Section \ref{sec:Geodesic-Flows-Symbolic},
we know that for every $\rho\in{\cal T}(S)$, the geodesics flow on
a smaller section $\Omega_{0}\subset T^{1}X_{\rho}$ conjugates the
suspension flow over a Markov shift $(\TMS,\s)=(\TMS_{A},\sigma_{A})$
with a unique (up to cohomologacy) locally H\"older continuous roof
function $\tau$. We want to point out again that the Markov shift
$(\TMS,\s)$ is constructed through the shape of fundamental domain.
Since type-preserving Fuchsian representations have the same shape
of fundamental domain, we know the suspension flow models for all
$\rho\in{\cal T}(S)$ have the same base space $(\TMS,\s)$ yet with
different roof functions. 

Let $\mathtt{\mathbf{P}}$ be the set of pressure zero locally H\"older
continuous functions on $\TMS$, that is,
\[
\mathbf{\mathbf{P}}:=\{\tau\in C(\TMS):\tau{\rm \ is\ locally\ H\ddot{o}lder,\ }P(-\tau)=0\}.
\]
In the following, we will discuss the relations between ${\cal T}(S)$
and $\mathtt{\mathbf{P}}$. Notice that since ${\cal T}(S)$ is composed
by representations in $\psl(2,\R)$, it inherits a natural analytic
structure from $\psl(2,\R)$ (see \cite{Hamenstadt:2003wn} for more
details). The following proposition indicates that there exists an
analytic thermodynamic mapping $\Phi:{\rm Hom_{tp}^{F}}(\pi_{1}(S),{\rm PSL}(2,\R))\to\mathbf{\mathbf{P}}$. 

\begin{prop}
[Thermodynamic Mapping]\label{prop:thermodynamic_map} Let $0<\ep\ll1$
and $\{\rho_{t}\}_{t\in(-\ep,\ep)}\subset{\cal T}(S)$ be an analytic
one-parameter family in ${\cal T}(S)$, then $\Phi(\{\rho_{t}\})=\{\tau_{t}\}\subset\mathbf{\mathbf{P}}$
is an analytic one-parameter family in $\mathbf{\mathbf{P}}$.\textcolor{red}{{} }
\end{prop}

\begin{proof}
We first notice that if $\{\rho_{t}\}\subset{\cal T}(S)$ is analytic
then the boundary map (derived in Theorem \ref{thm:Fenchel-Neilsen-Thm})
${\rm b}_{t}:\vbdy\H=\L(\rho_{0}(\pi_{1}S))\to\L(\rho_{t}(\pi_{1}S))=\vbdy\H$
is real analytic (see \cite[Section 2]{McMullen:2008eh} or \cite[Proposition 4.1]{Bridgeman:2017jy}).
For the completeness, we summarize the proof of this fact. The idea
of \cite[Section 2]{McMullen:2008eh}, as well as \cite[Proposition 4.1]{Bridgeman:2017jy},
is a complex analytic approach, namely, using holomorphic motions
and the $\lambda-$lemma. 

Let us denote $QF(S)$ the space conjugacy classes of quasi-Fuchsian
(i.e., the limit set is a Jordan curve) representations of $\pi_{1}(S)\to{\rm \psl(2,\mathbb{C})}$.
Recall that $QF(S)$ is an open neighborhood of ${\cal T}(S)$ in
the $\psl(2,\mathbb{C})-$character variety of $\pi_{1}(S)$. Let
$\rho_{t}$ vary in $QF(S)$, then there exist an embedding $\overline{{\rm b}}_{t}:\vbdy\H\to\L(\rho_{t}(\pi_{1}S)\subset\hat{\C}$.
Notice that $\rho_{0}\in{\cal T}(S)$ is fixed. It is clear that if
$\xi\in\L(\rho_{0}(\pi_{1}S))$ is fixed by a nontrivial element $\rho_{0}(\g)$,
then $b_{t}(\xi)$ varies holomorphically. Thus by Slodkowski's generalized
$\lambda-$lemma (cf. \cite{Slodkowski:1991hu}), we know $\overline{{\rm b}}_{t}$
varies complex analytically when $\rho_{t}$ varies in $QF(S)$; hence,
${\rm b}_{t}(={\rm \overline{b}_{t}})$ varies real analytically when
$\rho_{t}$ varies in ${\cal T}(S)$. 

To see $\{\tau_{t}\}$ is real analytic, by definition 
\[
\tau_{t}(x)=B_{\xi_{t}(x)}(o,\rho_{t}(g_{x_{0}})o)
\]
 where $x=x_{0}...$, and $\xi_{t}={\rm b}_{t}\circ\xi_{0}:\TMS\to\Lambda(\rho_{t}(\pi_{1}S))$.
Recall that in the disk model, we know $B_{\xi}(x,y)=\ln(\frac{1-|y|^{2}}{|\xi-y|^{2}}\frac{|\xi-x|^{2}}{1-|x|^{2}})$.
Thus we have, without loss of generality, taking $o$ to be the origin,
\[
\tau_{t}(x)=B_{{\rm b}_{t}\circ\xi_{0}}(o,\rho_{t}(g_{x_{0}})o)=\ln\frac{1-|\rho_{t}(g_{x_{0}})o|^{2}}{|{\rm b}_{t}\circ\xi_{0}(x)-\rho_{t}(g_{x_{0}})o|^{2}}.
\]
 Since both $\rho_{t}$ and $b_{t}$ vary real analytically, from
the above expression we know $\tau_{t}$ also varies real analytically. 
\end{proof}

By Corollary \ref{cor:bdd-diff-loof-fcn} we know that $\tau_{\rho_{1}}$
is locally H\"older continuous and $|\Phi(\rho_{1})-\Phi(\rho_{0})|=|\tau_{\rho_{1}}-\tau_{\rho_{0}}|$
is bounded for all $\rho_{0},\rho_{1}\in{\cal {\cal T}}(S)$. Thus,
consider an analytic path $\rho_{t}\subset{\cal T}(S)$, and we write
out the analytic path $\tau_{t}=\Psi(\rho_{t})$, in terms of Taylor
expansion, $\tau_{t}=\tau_{0}+t\cdot\dot{\tau}_{0}+...$. We know
the perturbation $\dot{\tau}_{0}$ is a bounded locally H\"older
continuous function. Therefore, it is sufficient to consider $T_{\tau_{0}}\mathbf{P}$,
the corresponding tangent space of $T_{\rho_{0}}{\cal T}(S)$, as
\[
T_{\tau_{0}}\mathbf{P}:=\{f\in C(\TMS):\int_{\TMS}f\dd m_{-\tau_{0}}=0,f{\rm \ is\ locally\ H\ddot{o}lder\ and\ bounded}\}\subset{\rm Ker}D_{-\tau_{0}}P.
\]
Moreover, we are interested in the pressure norm $||\cdot||_{P}$
on $\mathbf{P}$ given by 
\[
||f||_{{\rm P}}:=\frac{{\rm Var}(f,m_{-\tau_{0}})}{\int\tau_{0}\dd m_{-\tau_{0}}}.
\]
Notice that this norm degenerates precisely when $f\sim0$. In the
theorem below, we prove that one can define the \textit{pressure metric}
$||\cdot||_{{\rm }}$ on ${\cal T}(S_{g,n})$ through $||\cdot||_{P}$:

\begin{thm}
[Theorem \ref{thm:pressure metric}] \label{thm:Proof of ThmD}Suppose
$0<\ep\ll1$ and $\rho_{t}\in{\cal T}(S_{g,n})$ is an analytic path
for $t\in(-\ep,\ep)$. Then ${\rm I}(\rho_{0},\rho_{t})$ is real
analytic and 
\[
||\dot{\rho}_{0}||_{{\rm }}^{2}:=||d\Psi(\dot{\rho}_{0})||_{{\rm P}}^{2}=\left.\frac{{\rm d}^{2}{\rm I}(\rho_{0},\rho_{t})}{{\rm d}t^{2}}\right|_{t=0}
\]
defines a Riemannian metric on ${\cal T}(S_{g,n})$.
\end{thm}

\begin{proof}
Follows Proposition \ref{prop:intersection-dynamics} and Proposition
\ref{prop:thermodynamic_map}, we know ${\rm I}(\rho_{0},\rho_{t})$
is real analytic. Thus, it is sufficient to show that $\left.\frac{{\rm d}^{2}{\rm I}(\rho_{0},\rho_{t})}{{\rm d}t^{2}}\right|_{t=0}=||d\Psi(\dot{\rho}_{0})||_{{\rm P}}^{2}$
and $||d\Psi(\dot{\rho}_{0})||_{{\rm P}}^{2}>0$ when $\dot{\rho}_{0}\neq0$. 

By Proposition \ref{prop:intersection-dynamics} and Proposition \ref{prop:thermodynamic_map},
we know 
\[
\left.\frac{{\rm d}^{2}{\rm I}(\rho_{0},\rho_{t})}{{\rm d}t^{2}}\right|_{t=0}=\left.\frac{{\rm d}^{2}}{{\rm d}t^{2}}(\frac{\int\tau_{t}\dd m_{-\tau_{0}}}{\int\tau_{0}\dd m_{-\tau_{0}}})\right|_{t=0}=\frac{\int\ddot{\tau}_{0}\dd m_{-\tau_{0}}}{\int\tau_{0}\dd m_{-\tau_{0}}}
\]
where $\tau_{t}=\Psi(\rho_{t})$. Moreover, by Corollary \ref{cor:derivative-pressure},
we know 
\[
0=\left.\frac{{\rm d}^{2}P(-\tau_{t})}{{\rm d}t^{2}}\right|_{t=0}=(D_{-\tau_{0}}P)(-\ddot{\tau}_{0})+(D_{-\tau_{0}}^{2}P)(-\dot{\tau}_{0})=-\int\ddot{\tau}_{0}\dd m_{-\tau_{0}}+{\rm Var}(-\dot{\tau}_{0},m_{-\tau_{0}})
\]
and ${\rm Var}(-\dot{\tau}_{0},m_{-\tau_{0}})=0$ if and only if $\dot{\tau}_{0}\sim0$. 

To see the non-degeneracy, suppose $\dot{\tau}_{0}\sim0$ and let
$h$ be any hyperbolic element. Then $l(\rho_{t}[h])=S_{m}\tau_{t}(x)$
for some $x\in{\rm Fix}^{m}$, and thus 
\[
\left.\frac{{\rm d}}{{\rm d}t}\right|_{t=0}l(\rho_{t}[h])=S_{m}\dot{\tau}_{0}(x)=0.
\]
Moreover, since ${\cal T}(S)$ is can be parametrized by finitely
many (simple) closed geodesics (cf., for example, \cite{Hamenstadt:2003wn}),
 $\left.\frac{{\rm d}}{{\rm d}t}\right|_{t=0}l(\rho_{t}[h])=0$ for
all $h$ is hyperbolic implies $\dot{\rho}_{0}=0$. Hence, we have
$||\dot{\rho}_{0}||_{{\rm }}^{2}:=||d\Psi(\dot{\rho}_{0})||_{{\rm P}}^{2}=\left.\frac{{\rm d}^{2}{\rm I}(\rho_{0},\rho_{t})}{{\rm d}t^{2}}\right|_{t=0}$
and $||\dot{\rho}_{0}||_{{\rm }}^{2}=0$ if and only $\dot{\rho}_{0}=0$
in $T_{\rho_{0}}{\cal T}(S_{g,n})$. 
\end{proof}

\subsection{The Pressure Metric and Manhattan Curves}

In this subsection, we will prove Theorem \ref{thm:pressure-metric-manhattan-curve},
which points out that one can recover the Thurston's Riemannian metric
through varying the Manhattan curves. Let $\{\rho_{t}\}\in{\cal T}(S_{g,n})$
be an analytic path, and ${\cal C}(\rho_{0},\rho_{t})$ be the Manhattan
curve of $\rho_{0},\rho_{t}$. By Theorem \ref{Thm:manhattan-strictly-convex},
we know ${\cal C}(\rho_{0},\rho_{t})$ is a real analytic curve. Thus
we can parametrize ${\cal C}(\rho_{0},\rho_{t})$ by writing ${\cal C}(\rho_{0},\rho_{t})=\{(s,\chi_{t}(s)):s\in[0,1]\}$
where $\chi_{t}(s)$ is a real analytic function. See Figure 6.1.

\begin{figure}

\includegraphics[scale=1.5]{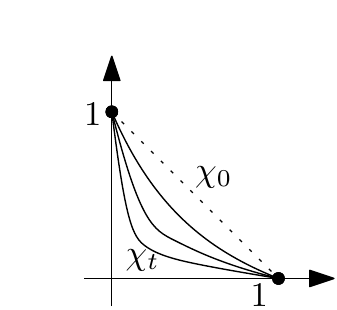}\caption{Manhattan Curves}

\end{figure}

\begin{thm}
[Theorem \ref{thm:pressure-metric-manhattan-curve}] Following the
above notations, we have for all $s\in(0,1)$
\[
\left.\frac{{\rm d}^{2}\chi_{t}(s)}{{\rm d}t^{2}}\right|_{t=0}=s(s-1)\cdot||\dot{\rho}_{0}||_{{\rm }}^{2}.
\]
\end{thm}

\begin{proof}
For a fixed $t$, ${\cal C}(\rho_{0},\rho_{t})$ can be identified
as 
\[
\{(a,b):a,b\geq0,a+b\neq0,{\rm and\ }P(-a\tau_{0}-b\tau_{t})=0\}=\{(s,\chi_{t}(s)):\ P(-s\tau_{0}-\chi_{t}(s)\tau_{t})=0,s\in[0,1]\}.
\]
For convenience, let us denote $\varphi_{t}(s):=-s\tau_{0}-\chi_{t}(s)\tau_{t}$.
Since when $t=0,$ $\chi_{0}$ is a straight line satisfying $s+\chi_{0}(s)=1$,
we have $\vp_{0}=-\tau_{0}.$ Thus, we know $\dot{\vp}_{0}=-\dot{\chi}_{0}\tau_{0}-\chi_{0}\dot{\tau}_{0}$
and $\ddot{\vp}_{0}=-\ddot{\chi_{t}}\tau_{0}-2\dot{\chi}_{0}\dot{\tau}_{0}-\chi_{0}\ddot{\tau}_{0}$.
By Corollary \ref{cor:derivative-pressure}, we get
\begin{alignat*}{1}
0=\left.\frac{{\rm d}}{{\rm d}t}P(\vp_{t})\right|_{t=0} & =\int\dot{\vp}_{0}\dd m_{\vp_{0}}\\
 & =\int-\dot{\chi}_{0}\tau_{0}-\chi_{0}\dot{\tau}_{0}\dd m_{-\tau_{0}}\\
 & =-\dot{\chi}_{0}(s)\int\tau_{0}\dd m_{-\tau_{0}}-\chi_{0}(s)\int\dot{\tau}_{0}\dd m_{-\tau_{0}}.
\end{alignat*}
Since $\int\dot{\tau}_{0}\dd m_{-\tau_{0}}=0$ (because $P(-\tau_{t})=0$)
and $\int-\tau_{0}\dd m_{-\tau_{0}}<0$, we have $\dot{\chi}_{0}(s)=0,\ \forall s\in[0,1]$.
Furthermore, by taking the second derivative of pressure (as in the
proof of Theorem \ref{thm:Proof of ThmD}) we get\textcolor{black}{
\begin{flalign*}
0=\left.\frac{{\rm d}^{2}}{{\rm d}t^{2}}P(\vp_{t})\right|_{t=0} & ={\rm Var}(\dot{\vp}_{0},m_{\vp_{0}})+\int\ddot{\vp}_{0}\dd m_{\vp_{0}}\\
 & ={\rm Var}(\underset{\overset{\shortparallel}{0}}{\underbrace{-\dot{\chi}_{0}}}\tau_{0}-\chi_{0}\dot{\tau}_{0},m_{-\tau_{0}})-\int\left(\ddot{\chi_{0}}\tau_{0}+2\underset{\overset{\shortparallel}{0}}{\underbrace{\dot{\chi_{0}}}}\dot{\tau}_{0}+\chi_{0}\ddot{\tau}_{0}\right)\dd m_{-\tau_{0}}\\
 & =(\chi_{0}(s))^{2}\cdot{\rm Var(}\dot{\tau}_{0},m_{-\tau_{0}})-\ddot{\chi_{0}}\int\tau_{0}\dd m_{-\tau_{0}}-\chi_{0}\int\ddot{\tau}_{0}\dd m_{-\tau_{0}}.
\end{flalign*}
Notice that $P(-\tau_{t})=0,$ similarly we have}
\[
0=\left.\frac{{\rm d}^{2}P(-\tau_{t})}{{\rm d}t^{2}}\right|_{t=0}=-\int\ddot{\tau}_{0}\dd m_{-\tau_{0}}+{\rm Var}(-\dot{\tau}_{0},m_{-\tau_{0}}).
\]
Therefore, we have 
\begin{alignat*}{1}
\ddot{\chi}_{0}(s) & =(\chi_{0}(s)^{2}-\chi_{0}(s))\frac{{\rm Var(}\dot{\tau}_{0},m_{-\tau_{0}})}{\int\tau_{0}\dd m_{-\tau_{0}}}\\
 & =\left((1-s)^{2}-(1-s)\right))\frac{{\rm Var(}\dot{\tau}_{0},m_{-\tau_{0}})}{\int\tau_{0}\dd m_{-\tau_{0}}}=(s^{2}-s)||\dot{\rho}_{0}||_{{\rm }}^{2}.
\end{alignat*}
\end{proof}

\bibliographystyle{amsalpha}
\bibliography{/Users/nyima/Dropbox/TEX/Bibtex/BIB}

\providecommand{\bysame}{\leavevmode\hbox to3em{\hrulefill}\thinspace}
\providecommand{\MR}{\relax\ifhmode\unskip\space\fi MR }
\providecommand{\MRhref}[2]{%
  \href{http://www.ams.org/mathscinet-getitem?mr=#1}{#2}
}
\providecommand{\href}[2]{#2}
\begin{thebibliography}{BCLS15}

\bibitem[Ahl06]{Ahlfors:2006ix}
Lars~V. Ahlfors, \emph{Lectures on quasiconformal mappings}, second ed.,
  University Lecture Series, vol.~38, American Mathematical Society,
  Providence, RI, 2006, With supplemental chapters by C. J. Earle, I. Kra, M.
  Shishikura and J. H. Hubbard. \MR{2241787}

\bibitem[AK42]{Ambrose:1942fj}
Warren Ambrose and Shizuo Kakutani, \emph{Structure and continuity of
  measurable flows}, Duke Math. J. \textbf{9} (1942), 25--42. \MR{0005800}

\bibitem[BCLS15]{Bridgeman:2013to}
Martin Bridgeman, Richard Canary, Fran{\c c}ois Labourie, and Andr{\'e}s
  Sambarino, \emph{{The pressure metric for {A}nosov representations}}, Geom.
  Funct. Anal. \textbf{25} (2015), no.~4, 1089--1179.

\bibitem[BCS18]{Bridgeman:2017jy}
Martin Bridgeman, Richard Canary, and Andr\'{e}s Sambarino, \emph{An
  introduction to pressure metrics for higher {T}eichm\"{u}ller spaces},
  Ergodic Theory Dynam. Systems \textbf{38} (2018), no.~6, 2001--2035.
  \MR{3833339}

\bibitem[BI06]{Barreira:2006fd}
Luis Barreira and Godofredo Iommi, \emph{{Suspension flows over countable
  Markov shifts}}, J. Stat. Phys. \textbf{124} (2006), no.~1, 207--230.

\bibitem[BS93]{Bishop:1991gz}
Christopher Bishop and Tim Steger, \emph{{Representation-theoretic rigidity in
  PSL(2, R)}}, Acta Mathematica \textbf{170} (1993), no.~1, 121--149.

\bibitem[Bur93]{Burger:1993wb}
Marc Burger, \emph{{Intersection, the {M}anhattan curve, and
  {P}atterson-{S}ullivan theory in rank {$2$}}}, Internat. Math. Res. Notices
  (1993), no.~7, 217--225.

\bibitem[CI18]{Cipriano:2018wz}
Italo Cipriano and Godofredo Iommi, \emph{{Time change for flows and
  thermodynamic formalism}}, arXiv.org (2018).

\bibitem[Ham03]{Hamenstadt:2003wn}
Ursula Hamenst\"{a}dt, \emph{Length functions and parameterizations of
  {T}eichm\"{u}ller space for surfaces with cusps}, Ann. Acad. Sci. Fenn. Math.
  \textbf{28} (2003), no.~1, 75--88. \MR{1976831}

\bibitem[IJT15]{Iommi:2015th}
Godofredo Iommi, Thomas Jordan, and Mike Todd, \emph{{Recurrence and transience
  for suspension flows}}, Israel J. Math. \textbf{209} (2015), no.~2, 547--592.

\bibitem[JKL14]{Jaerisch:2014js}
Johannes Jaerisch, Marc Kesseb\"ohmer, and Sanaz Lamei, \emph{{Induced
  topological pressure for countable state Markov shifts}}, Stoch. Dyn.
  \textbf{14} (2014), no.~2, 1350016--1350031.

\bibitem[Kao18]{Kao:2018th}
Lien-Yung Kao, \emph{{Manhattan Curves for Hyperbolic Surfaces with Cusps}},
  Ergodic Theory Dynam. Systems (2018), 1--32.

\bibitem[Kap09]{Kapovich:2009fk}
Michael Kapovich, \emph{{Hyperbolic manifolds and discrete groups}}, Modern
  Birkh\"auser Classics, Birkh\"auser Boston, Inc., Boston, MA, Boston, 2009.

\bibitem[Kem11]{Kempton:2011hs}
Tom Kempton, \emph{{Thermodynamic formalism for suspension flows over countable
  Markov shifts}}, Nonlinearity \textbf{24} (2011), no.~10, 2763--2775.

\bibitem[Kim01]{Kim:2001km}
Inkang Kim, \emph{Marked length rigidity of rank one symmetric spaces and their
  product}, Topology \textbf{40} (2001), no.~6, 1295--1323. \MR{1867246}

\bibitem[LS08]{Ledrappier:2008wq}
Fran{\c c}ois Ledrappier and Omri Sarig, \emph{{Fluctuations of ergodic sums
  for horocycle flows on {$\Bbb Z^d$}-covers of finite volume surfaces}},
  Discrete Contin. Dyn. Syst. \textbf{22} (2008), no.~1-2, 247--325.

\bibitem[McM08]{McMullen:2008eh}
Curtis McMullen, \emph{{Thermodynamics, dimension and the Weil-Petersson
  metric}}, Invent. Math. \textbf{173} (2008), no.~2, 365--425.

\bibitem[MU01]{Mauldin:2001dn}
Daniel Mauldin and Mariusz Urba\'nski, \emph{{Gibbs states on the symbolic
  space over an infinite alphabet}}, Israel J. Math. \textbf{125} (2001),
  no.~1, 93--130.

\bibitem[MU03]{Mauldin:2003dn}
\bysame, \emph{{Graph directed Markov systems}}, Cambridge Tracts in
  Mathematics, vol. 148, Cambridge University Press, Cambridge, Cambridge,
  2003.

\bibitem[OP04]{Otal:2004fn}
Jean-Pierre Otal and Marc Peign{\'e}, \emph{{Principe variationnel et groupes
  kleiniens}}, Duke Math. J. \textbf{125} (2004), no.~1, 15--44.

\bibitem[PPS15]{Pollicott:2012ud}
Fr{\'e}d{\'e}ric Paulin, Mark Pollicott, and Barbara Schapira,
  \emph{{Equilibrium states in negative curvature}}, no. 373, Ast\'erisque,
  2015.

\bibitem[PS16]{Pollicott:2014uk}
Mark Pollicott and Richard Sharp, \emph{Weil-{P}etersson metrics, {M}anhattan
  curves and {H}ausdorff dimension}, Math. Z. \textbf{282} (2016), no.~3-4,
  1007--1016.

\bibitem[Sar99]{Sarig:1999wo}
Omri Sarig, \emph{{Thermodynamic formalism for countable Markov shifts}},
  Ergodic Theory Dynam. Systems \textbf{19} (1999), no.~6, 1565--1593.

\bibitem[Sar01]{Sarig:2001bj}
\bysame, \emph{{Phase transitions for countable Markov shifts}}, Comm. Math.
  Phys. \textbf{217} (2001), no.~3, 555--577.

\bibitem[Sar03]{Sarig:2003hl}
\bysame, \emph{{Existence of Gibbs measures for countable Markov shifts}},
  Proc. Amer. Math. Soc. \textbf{131} (2003), no.~6, 1751--1758 (electronic).

\bibitem[Sar09]{Sarig:2009wta}
\bysame, \emph{{Lecture notes on thermodynamic formalism for topological Markov
  shifts}}, 2009.

\bibitem[Sav98]{Savchenko:1998fh}
Sergei~V. Savchenko, \emph{{Special flows constructed from countable
  topological Markov chains}}, Funktsional. Anal. i Prilozhen. \textbf{32}
  (1998), no.~1, 40--53, 96.

\bibitem[Sha98]{Sharp:1998if}
Richard Sharp, \emph{{The Manhattan curve and the correlation of length spectra
  on hyperbolic surfaces}}, Mathematische Zeitschrift \textbf{228} (1998),
  no.~4, 745--750.

\bibitem[Slo91]{Slodkowski:1991hu}
Zbigniew Slodkowski, \emph{Holomorphic motions and polynomial hulls}, Proc.
  Amer. Math. Soc. \textbf{111} (1991), no.~2, 347--355. \MR{1037218}

\bibitem[Tuk72]{Tukia:1972ti}
Pekka Tukia, \emph{On discrete groups of the unit disk and their isomorphisms},
  Ann. Acad. Sci. Fenn. Ser. A I (1972), no.~504, 45 pp. (errata insert).
  \MR{0306487}

\bibitem[Tuk73]{Tukia:1973uw}
\bysame, \emph{Extension of boundary homeomorphisms of discrete groups of the
  unit disk}, Ann. Acad. Sci. Fenn. Ser. A I (1973), no.~548, 16. \MR{0338358}

\bibitem[Wol86]{Wolpert:1986wv}
Scott~Andrew Wolpert, \emph{{Thurston's Riemannian metric for Teichm{\"u}ller
  space}}, Journal of Differential Geometry \textbf{23} (1986), no.~2,
  143--174.

\end{thebibliography}

\end{document}